\documentclass[a4paper]{amsart}

\usepackage{latexsym}
\usepackage{a4wide}
\usepackage{amscd}
\usepackage{graphics}
\usepackage{amsmath}
\usepackage{amssymb}
\usepackage{bm}
\usepackage{mathrsfs}
\usepackage{enumerate}
\usepackage{hyperref}
\usepackage{pgf,tikz}
\usepackage[utf8]{inputenc}

\numberwithin{equation}{section} 

\newtheorem{theorem}{Theorem}[section]
\newtheorem{corollary}[theorem]{Corollary}
\newtheorem{lemma}[theorem]{Lemma}
\newtheorem{proposition}[theorem]{Proposition}
 \theoremstyle{definition} 
 \newtheorem{definition}[theorem]{Definition}
\newtheorem{example}[theorem]{Example}
\newtheorem{examples}[theorem]{Examples}

 \newtheorem{remark}[theorem]{Remark}

\newcommand{\R}{\mathbb{R}}	
\newcommand{\N}{\mathbb{N}} 

\renewcommand{\phi}{\varphi}
\newcommand\restr[1]{\raisebox{-.5ex}{$|$}_{#1}}	
\newcommand{\dx}{\,\mathrm{d}x}	
\newcommand{\dy}{\,\mathrm{d}y}	
\newcommand{\ds}{\,\mathrm{d}S}	
\newcommand{\nnu}{\bm{\nu}}	
\newcommand{\norm}[1]{\left\lVert #1 \right\lVert}
\newcommand{\abs}[1]{\left| #1 \right|}
\newcommand{\DCK}{\mathcal{D}^{1,2}(\overline{\R^N_+}\setminus K)}	
\newcommand{\DC}{\mathcal{D}^{1,2}(\overline{\R^N_+})}	
\newcommand{\HO}[2]{H^1_{0,#2}(#1)}

\newcommand{\sub}{\subseteq}

\DeclareMathOperator{\supp}{\mathrm{supp}}
\DeclareMathOperator{\Capa}{\mathrm{Cap}}
\DeclareMathOperator{\capa}{\mathrm{cap}}
\DeclareMathOperator{\Span}{\mathrm{span}}
\DeclareMathOperator{\dist}{\mathrm{dist}}
\DeclareMathOperator{\dive}{\mathrm{div}}

\newenvironment{bvp}{\left\{\begin{aligned}  }{\end{aligned}\right.}

\begin{document}

\title[Eigenvalues with moving mixed boundary conditions]{Eigenvalues of the Laplacian with moving mixed boundary conditions: the case of disappearing Dirichlet region}

\author[V. Felli]{Veronica Felli}
\address{Veronica Felli
\newline \indent Dipartimento di Matematica e Applicazioni, Universit\`a di Milano–Bicocca
\newline\indent Via Cozzi 55, 20125 Milano, Italy}
\email{veronica.felli@unimib.it}

\author[B. Noris]{Benedetta Noris}
\address{Benedetta Noris
\newline \indent Dipartimento di Matematica, Politecnico di Milano\newline\indent Piazza Leonardo da Vinci 32, 20133 Milano, Italy}
\email{benedetta.noris@polimi.it}

\author[R. Ognibene]{Roberto Ognibene}
\address{Roberto Ognibene
\newline \indent Dipartimento di Matematica e Applicazioni, Universit\`a di Milano–Bicocca
\newline\indent Via Cozzi 55, 20125 Milano, Italy}
\email{roberto.ognibene@unimib.it}


\date{}

\begin{abstract}
  	In this work we consider the homogeneous Neumann eigenvalue problem for the Laplacian on a bounded Lipschitz domain and a singular perturbation of it, which consists in prescribing zero Dirichlet boundary conditions on a small subset of the boundary. 
 We first describe the sharp asymptotic behaviour of a perturbed	eigenvalue, in the case in which it is converging to a simple eigenvalue of the limit Neumann problem. The first term in the asymptotic expansion turns out to depend on the Sobolev capacity of the subset where the perturbed eigenfunction is vanishing. Then we focus on the case of Dirichlet boundary conditions imposed on a subset which is scaling to a point;
by a blow-up analysis for the capacitary potentials, we detect the vanishing order of the Sobolev capacity of such shrinking Dirichlet boundary  portion.
\end{abstract}
	
	\maketitle
	

\medskip 


\section{Introduction}\label{sec:intr}

The present work concerns the eigenvalue problem for the Laplacian with mixed Dirichlet-Neumann homogeneous boundary conditions. More in particular, our attention is devoted to the study of the behaviour of an eigenvalue of the mixed problem when the region where Dirichlet boundary conditions are prescribed is disappearing, in a suitable sense that will be specified later. 
Actually, the methods developed in the present paper to derive eigenvalue asymptotics under perturbed mixed boundary conditions turn out to be quite flexible and capable of treating also more general kinds of perturbation, e.g. the eigenvalue problem for the Neumann-Laplacian with a shrinking hole in the interior of the domain  where homogeneous Dirichlet boundary conditions are assigned.

Let us introduce some basic assumptions and the functional setting. Let $\Omega\subseteq \R^N$ (with $N\geq 2$) be an open, bounded, Lipschitz and connected set and let $K\subseteq \overline{\Omega}$ be compact. Let $c\in L^\infty(\R^N)$ be such that 
\begin{equation}\label{eq:c_assumption}
c(x)\geq c_0>0 \quad\text{a.e. in }\R^N,\quad\text{for some }c_0\in\R.
\end{equation}
We define the bilinear form  $q=q_c:H^1(\Omega)\times H^1(\Omega)\to\R$ as 
\begin{equation}\label{eq:scalar_prod}
 q(u,v):=\int_{\Omega}(\nabla u\cdot\nabla v+cuv)\dx \quad\text{for any $u,v\in H^1(\Omega)$}.
\end{equation}
For simplicity of notation we denote by $q(\cdot)$ also the quadratic form corresponding to \eqref{eq:scalar_prod}, i.e. $q(u)=q(u,u)$. Thanks to assumption \eqref{eq:c_assumption}, the square root of the quadratic form $q(\cdot)$ is a norm on $H^1(\Omega)$, equivalent to the standard one
\[
\norm{u}_{H^1(\Omega)}=\left(\int_{\Omega}(\abs{\nabla u}^2+u^2)\dx \right)^{1/2}.
\]
We also introduce the Sobolev space $\HO{\Omega}{K}$ defined as the closure in $H^1(\Omega)$ 
of $C_c^\infty(\overline{\Omega}\setminus K)$.
We observe that, if $\partial\Omega$ is smooth and $K$ is a regular submanifold of $\partial\Omega$, 
the space $\HO{\Omega}{K}$ can be characterized as
\[
\HO{\Omega}{K}=\{u\in H^1(\Omega)\colon u=0~\text{on }K \},
\]
where $u=0$ on $K$, for functions in $H^1(\Omega)$, is meant in the trace sense, see \cite{Bernard2011}. Defining $q_K$ as the restriction of the form $q$ to $\HO{\Omega}{K}$, we say that $\lambda\in\R$ is an \emph{eigenvalue} of $q_K$ if there exists $u\in\HO{\Omega}{K}$, $u\not\equiv 0$, called \emph{eigenfunction}, such that
\begin{equation}\label{eq:weak_mixed}
q_K(u,v)=\lambda (u,v)_{L^2(\Omega)}\quad\text{for all }v\in \HO{\Omega}{K},
\end{equation}
where $(\cdot,\cdot)_{L^2(\Omega)}$ is the usual scalar product in $L^2(\Omega)$. From classical spectral theory we have that problem \eqref{eq:weak_mixed} admits a diverging sequence of positive eigenvalues
\[
0<\lambda_1(\Omega;K)< \lambda_2(\Omega;K)\leq \cdots\leq \lambda_n(\Omega;K)\leq \cdots,
\]
where each one is repeated as many times as its multiplicity. Moreover, we denote by $(\phi_n(\Omega;K))_n$ a sequence of eigenfunctions, which we choose so that it forms an orthonormal family in $L^2(\Omega)$. Hereafter we denote, for any integer $n\in\N_*$,
\begin{equation}\label{eq:def_neu_eigen}
\lambda_n:=\lambda_n(\Omega;\emptyset),\quad \phi_n:=\phi_n(\Omega;\emptyset),
\end{equation}
where $\N_*:=\N\setminus\{0\}$. 
We notice that the connectedness of the domain $\Omega$ is not a restrictive assumption, since the spectrum of $q_K$ in a non connected domain is the union of the spectra on the single connected components. We also point out that the assumption \eqref{eq:c_assumption} is not substantial and it can be dropped, since, up to a translation of the spectrum, we can recover a coercive form as in \eqref{eq:scalar_prod}. Besides, we notice that, in the particular case when $K$ is the empty set and $c(x)\equiv c>0$, $(\lambda_n(\Omega;\emptyset)-c)_n$ coincides with the sequence of eigenvalues of the standard Laplacian with homogeneous Neumann boundary condition. 
If  $K$ is  smooth (e.g. if $K$ is the closure of a smooth open set of $\R^N$ or
a regular submanifold of $\partial\Omega$), problem \eqref{eq:weak_mixed} admits the following classical formulation
\begin{equation}\label{eq:mixed}
\begin{bvp}
-\Delta u+cu &=\lambda u, &&\text{in }\Omega\setminus K, \\
u&=0 , &&\text{on } K,\\
\frac{\partial u}{\partial \nnu}&=0, &&\text{on } \partial\Omega\setminus K.
\end{bvp}
\end{equation}
When $K\sub \partial\Omega$, \eqref{eq:mixed} is an elliptic problem with mixed Dirichlet-Neumann homogeneous boundary conditions and one can interpret the spectrum $(\lambda_n(\Omega;K))_n$ as the square roots of the frequencies of oscillation of an elastic, vibrating membrane, whose boundary is clamped on $K$ and free in the rest of $\partial\Omega$. 

In this paper we 
start from the unperturbed situation corresponding to
the Neumann eigenvalue problem, i.e. the case $K=\emptyset$, and then we introduce a singular perturbation of it, which consist in considering a ``small'', nonempty $K\sub\overline{\Omega}$ and zero Dirichlet boundary conditions on it. Our aim is to study the eigenvalue variation due to this perturbation and to find the sharp asymptotics of the perturbed eigenvalue, in the limit when $K$ is ``disappearing'' as a  function of a certain parameter.

A detailed analysis in dimension 2 has been performed in \cite{Gadylshin1993}, where
Ga\-dyl’shin
investigated the case in which the perturbing set is a segment of length $\epsilon\to 0^+$ contained in the boundary of the domain: a (possibly multiple) eigenvalue of the limit Neumann problem is considered and the full asymptotic expansion of the perturbed eigenvalues is provided, see Theorems 1 and 2 in \cite{Gadylshin1993}. These expansions strongly depend on the vanishing order of the limit eigenfunctions in the point of the boundary where the segment is concentrating. In \cite{Gadylshin1993} a complete pointwise expansion of the perturbed eigenfunctions is provided as well. Moreover, in \cite{Gadyl'shin1992},  Gadyl’shin considered, again in dimension 2, the complementary problem, i.e. the case in which the portion of the boundary where Neumann conditions are prescribed is vanishing. For this problem, in \cite{Gadyl'shin1992}, the splitting of a multiple, limit eigenvalue is proved, together with a full asymptotic expansion of the perturbed eigenvalues. In the same framework, in \cite{Felli2018} simple limit eigenvalues were considered, with the derivation of  a more explicit expression of the coefficient of the  leading term in the expansion stated in \cite{Gadyl'shin1992}. Besides, 
\cite{Felli2018} contains a blow-up convergence result for the scaled, perturbed eigenfunction and  some applications to eigenvalue problems for operators with Aharonov-Bohm potentials.

On the other hand, for arbitrary dimensions,  \cite{Courtois1995} and \cite{AFHL} treated the eigenvalue problem for the Dirichlet-Laplacian under perturbations consisting  in making a hole in the interior of the domain and letting it shrink. In particular, in \cite{AFHL} (and in its nonlocal counterpart \cite{AFN}), where arguments and techniques that inspired the ones developed in the present work were introduced, the sharp asymptotic behaviour of perturbed eigenvalues is described. We also mention \cite{Colorado2003}, which concerns the spectral stability of the first eigenvalue, in both  cases of shrinking  Neumann and Dirichlet part, and \cite{Leonori2018}, where again the first eigenvalue of the mixed problem is considered, but in a nonlocal framework. 

\section{Statement of the main results}
In this section we give the basic definitions used in our work and we present our main results. Let us recall that,  throughout the paper, $\Omega$ denotes an open, bounded, Lipschitz and connected subset of $\R^N$, where $N\geq 2$.

As in 
\cite{AFHL,AFN,Courtois1995}, the quantity that measures the ``smallness'' of the perturbation set $K\sub\overline{\Omega}$, which is suitable for the development of an eigenvalue  stability  theory for our problem, is a notion of capacity, as defined below.
\begin{definition}\label{def:capacity}
	Let $K\subseteq \overline{\Omega}$ be compact. We define the \emph{relative Sobolev capacity} of $K$ in $\overline{\Omega}$ as follows
	\begin{equation*}
	\Capa_{\bar{\Omega}}(K):=\inf\left\{ \int_{\Omega}(\abs{\nabla u}^2+u^2)\dx\colon u\in H^1(\Omega),~u-1\in \HO{\Omega}{K} \right\}.
	\end{equation*}
\end{definition}
We refer to \cite{Arendt2003} for the mathematical description of this set function. A first taste of the fact that 
the relative Sobolev capacity defined above is a good perturbation parameter
for our purposes is given by Proposition \ref{prop:Kequiv}, which states that the space $\HO{\Omega}{K}$ coincides with $H^1(\Omega)$ if and only if $\Capa_{\bar{\Omega}}(K)=0$. This means that zero capacity sets are negligible for $H^1$ functions. Furthermore, the following theorem yields continuity of perturbed eigenvalues when $\Capa_{\bar{\Omega}}(K)\to 0$.
\begin{theorem}\label{thm:cont_eigen}
	Let $K\subseteq\overline{\Omega}$ be compact and let $\lambda_n(\Omega;K)$ be an eigenvalue of problem \eqref{eq:weak_mixed} for some $n\in\N_*$. Let also $\lambda_n$ be as in \eqref{eq:def_neu_eigen}. Then there exist $C>0$ and $\delta>0$ (independent of $K$) such that, if $\Capa_{\bar{\Omega}}(K)<\delta$, then
	\[
	0\leq \lambda_n(\Omega;K)-\lambda_n\leq C \left(\Capa_{\bar{\Omega}}(K)\right)^{1/2}.
	\]
\end{theorem}
We observe that the left inequality is an easy consequence of the Min-Max variational characterization of the eigenvalues, namely 
\begin{equation}\label{eq:minmax}
\lambda_n(\Omega;K)=\min\left\{ \max_{u\in V_n}\frac{q(u)}{\norm{u}_{L^2(\Omega)}^2}\colon V_n\subseteq\HO{\Omega}{K} \text{ $n$-dimensional subspace} \right\}.
\end{equation}
In order to state the first  main result of this paper, we introduce the following notion of convergence of sets.
\begin{definition}\label{def:concentr}
	Let $K\subseteq \overline{\Omega}$ be compact and let $\{K_\epsilon\}_{\epsilon>0}$ be a family of compact subsets of $\overline{\Omega}$. We say that $K_\epsilon$ is \emph{concentrating} at $K$, as $\epsilon\to 0$, if for any open set $U\subseteq \R^N$ such that $K\subseteq U$ there exists $\epsilon_U>0$ such that $K_\epsilon\subseteq U$ for all $\epsilon\in(0,\epsilon_U)$.
\end{definition}
We observe that the ``limit'' set of a concentrating family is not unique. Indeed, if $K_\epsilon$ is concentrating at $K$, then it is also concentrating at any compact set $\tilde{K}$ such that $K\subseteq \tilde{K}\subseteq \overline{\Omega}$. Nevertheless, in the cases considered in the present paper (e.g.  when the limit set has zero capacity) this notion of convergence is the one that ensures the continuity of the capacity (see Proposition \ref{prop:cont_cap}); furthermore, it is related to the
convergence of sets in the sense of Mosco, see \cite{AFHL}.
An example of family of concentrating sets is given by a decreasing family of compact sets, see Example \ref{ex:concentr_sets}.

In order to sharply describe the eigenvalue variation, the following definition of capacity associated to a $H^1$-function plays a fundamental role.
\begin{definition}\label{def:f-capacity}
	For any $f\in H^1(\Omega)$ and $K\subseteq\overline{\Omega}$ compact, we define the \emph{relative Sobolev $f$-capacity} of $K$ in $\overline{\Omega}$ as follows
	\begin{equation}\label{eq:def_f_capacity}
	\Capa_{\bar{\Omega},c}(K,f):=\inf\left\{ q(u)\colon u\in H^1(\Omega),~u-f\in \HO{\Omega}{K} \right\}.
	\end{equation}
\end{definition}
We remark that, if $K\sub\partial\Omega$, the above definition actually only depends on the trace of $f$ on $\partial\Omega$, which belongs to $H^{1/2}(\partial\Omega)$, and in particular on its values on $K$, if $K$ is regular. Moreover one can prove that, if a family of compact sets $K_\epsilon\sub\overline{\Omega}$ is concentrating to a compact $K\sub\overline{\Omega}$ such that $\Capa_{\bar{\Omega}}(K)=0$, then $\Capa_{\bar{\Omega},c}(K_\epsilon,f)\to 0$ as $\epsilon\to 0$ for all $f\in H^1(\Omega)$, see Proposition \ref{prop:cont_cap} and Remark \ref{rmk:zero_cap_equiv}. 

Hereafter we assume $n_0\in\N_*$ to be such that 
\begin{equation}\label{eq:lambda_0}
\lambda_0:=\lambda_{n_0} \text{ is simple}
\end{equation}
and we denote as 
\begin{equation}\label{eq:phi0}
\phi_0:=\phi_{n_0},
\end{equation}
a corresponding $L^2(\Omega)$-normalized eigenfunction. Our first main result is 
the following sharp asymptotic expansion of the eigenvalue variation.
\begin{theorem}\label{thm:sharp_asymp}
	Let $\{K_\epsilon\}_{\epsilon>0}$ be a family of compact subsets of $\overline{\Omega}$ concentrating to $K\subseteq\overline{\Omega}$ compact such that $\Capa_{\bar{\Omega}}(K)=0$. Let $\lambda_0$, $\phi_0$ be as in \eqref{eq:lambda_0}, \eqref{eq:phi0} respectively and let $\lambda_\epsilon:=\lambda_{n_0}(\Omega;K_\epsilon)$. Then
	\[
	\lambda_\epsilon-\lambda_0= \Capa_{\bar{\Omega},c}(K_\epsilon,\phi_0)+o(\Capa_{\bar{\Omega},c}(K_\epsilon,\phi_0))
	\]
	as $\epsilon\to 0$.
\end{theorem}
In order to give
some relevant examples of explicit expansions, 
 in the last part of the present work we provide the sharp asymptotic behavior of the function $\epsilon\mapsto \Capa_{\bar{\Omega},c}(K_\epsilon,\phi_0)$ appearing above,  in a particular case.
 More precisely, we consider a family $\{K_\epsilon\}_{\epsilon>0}\sub \overline{\Omega}$ which is concentrating at a point $\bar{x}\in\overline{\Omega}$ in an appropriate way, that resembles the situation where a fixed set is being scaled and it is therefore maintaining the same shape while shrinking to the point. 
Hereafter we illustrate these results by distinguishing the cases $\bar{x}\in\partial\Omega$ and $\bar{x}\in\Omega$.
Without losing generality we can assume that $\bar{x}=0$. 
We perform  this analysis under the assumption $N\geq 3$, since
a detailed study of the case $N=2$, with $K_\epsilon,K\sub\partial\Omega$, has been already pursued in \cite{Gadylshin1993}; nevertheless, our method, which is based on  a
blow-up analysis for the capacitary potentials, could be adapted to the $2$-dimensional case by using a logarithmic Hardy inequality to derive energy estimates, instead of the Hardy-type inequality of Lemma \ref{lemma:hardy}, which does not hold in dimension~2.

\subsection{Sets scaling to a boundary point}\label{sec:sets-scal-bound}

We first focus on the case in which the perturbing compact sets $K_\epsilon\sub \overline{\Omega}$ are concentrating to a point of the boundary of $\Omega$, which, up to a translation, can be assumed to be the origin. In this situation, we assume that the boundary $\partial \Omega$ is of class $C^{1,1}$ in a neighbourhood of $0\in\partial\Omega$, namely 
\begin{equation}\label{eq:hp_boundary_1}
\begin{aligned}
\text{there exists }&r_0>0~\text{and}~g\in C^{1,1}(B_{r_0}')~ \text{such that }\\
B_{r_0}\cap \Omega&=\{x\in B_{r_0}\colon x_N>g (x') \}, \\
B_{r_0}\cap \partial\Omega&=\{x\in B_{r_0}\colon x_N=g (x') \},
\end{aligned}
\end{equation}
where 
$B_{r_0}=\{x=(x_1,x_2,\dots,x_N)\in\R^N:|x|<r_0\}$ is the ball in $\R^N$ centered at the origin  with radius $r_0$, 
$x'=(x_1,\dots,x_{N-1})$ and $B_{r_0}'=\{(x',x_N)\in B_{r_0}:x_N=0\}$. It is not restrictive to assume that $\nabla g (0)=0$, i.e. that $\partial\Omega$ is tangent to the coordinate hyperplane $\{x_N=0\}$ in the origin. 
Let us introduce the following class
of diffeomorphisms that ``straighten'' the boundary near $0$:
\begin{align}\label{eq:diffeo_class}
	\mathcal{C}:=\{ \Phi:\mathcal U\to B_R\colon &
\mathcal U\text{ is an open neighbourhood of $0$, }R>0,\\
\notag&\Phi\text{ is a diffeomorphism  of class $C^{1,1}(\mathcal U;B_{R})$, }\Phi(0)=0,
\\
\notag&	J_\Phi(0)=I_N, \ \Phi(\mathcal U\cap \Omega)= \R^N_+\cap B_R\text{ and }
\Phi(\mathcal U\cap \partial \Omega) =B_R'\},
\end{align}
where  $\R^N_+:=\{(x_1,\dots,x_N)\in \R^N\colon  x_N>0\}$ and $I_N$ is the identity $N\times N$ matrix.
Let us assume that, for any $\epsilon\in (0,1)$, $K_\epsilon\sub \overline{\Omega}$ is a compact set and the family $\{K_\epsilon\}_{\epsilon}$ satisfies the following properties:
\begin{gather}
\text{there exists }M\sub\overline{\R^N_+}~\text{compact such that }  \Phi(K_\epsilon)/\epsilon\sub M\quad\text{for all }\epsilon\in (0,1), \label{eq:hp_blow_up_1_intr} \\
\begin{gathered}
\text{there exists }K\sub\overline{\R_+^N}~\text{compact such that }\\
\R^N\setminus (\Phi(K_\epsilon)/\epsilon)\to \R^N\setminus K\quad\text{in the sense of Mosco, as }\epsilon\to 0,
\end{gathered}
\label{eq:hp_blow_up_2_intr}
\end{gather}
for some $\Phi\in\mathcal{C}$, where $\Phi(K_\epsilon)/\epsilon:=\{x/\epsilon\colon x\in \Phi(K_\epsilon)  \}$. With reference to \cite{Daners2003,Mosco1969}, we recall below the definition of convergence of sets in the sense of Mosco.
\begin{definition}\label{def:mosco_D}
	Let $\epsilon\in(0,1)$ and let $U_\epsilon,U\sub\R^N$ be open sets. We say that $U_\epsilon$ \emph{is converging to $U$ in the sense of Mosco} as $\epsilon\to 0$ if the following  two properties hold:
	\begin{enumerate}[(i)]
		\item the weak limit points (as $\epsilon\to 0$) in $H^1(\R^N)$ of every family of functions $\{u_\epsilon\}_{\epsilon}\subseteq H^1(\R^N)$, such that $u_\epsilon\in H^1_0(U_\epsilon)$ for every $\epsilon>0$, belong to $H^1_0(U)$;
		\item for every $u\in H^1_0(U)$ there exists a family $\{u_\epsilon\}_\epsilon\subseteq H^1(\R^N)$ such that $u_\epsilon\in H^1_0(U_\epsilon)$ for every $\epsilon>0$ and $u_\epsilon\to u$ in $H^1(\R^N)$, as $\epsilon\to 0$.
                \end{enumerate}
	We may also say that $H^1_0(U_\epsilon)$ is converging to $H^1_0(U)$ in the sense of Mosco.
\end{definition}

	In order to clarify hypotheses \eqref{eq:hp_blow_up_1_intr} and \eqref{eq:hp_blow_up_2_intr} we adduce below a bunch of examples in which they hold for subsets $K_\epsilon$ of $\partial\Omega$. 

\begin{examples}\label{ex:K_epsilon}\quad

\begin{enumerate}[(i)]
\item 
The easiest case is when $\partial\Omega$ is flat in a neighbourhood of the origin and 
	\begin{equation*}
	K_\epsilon:=\epsilon K=\{\epsilon x\colon x\in K  \},
	\end{equation*}
	for a certain fixed $K\sub\R^{N-1}$ compact. Here we can choose as $\Phi$ the identity so that $\Phi(K_\epsilon)/\epsilon\equiv K$, which clearly satisfies both  hypotheses  \eqref{eq:hp_blow_up_1_intr} and \eqref{eq:hp_blow_up_2_intr}.
\item 
 Another interesting example (always in the case of flat boundary) is when $\Phi$ is the identity and $K_\epsilon/\epsilon$ is a perturbation of a compact set. More precisely, let $K_1,K_2\sub\R^{N-1}$ be two compact sets containing the origin  and let $f\colon (0,1)\to (0,+\infty)$ be such that $f(s)/s\to 0$ as $s\to 0$. If we consider
	\begin{equation*}
		K_\epsilon:=\epsilon K_1+f(\epsilon) K_2=\{\epsilon x+f(\epsilon) y\colon x\in K_1,~y\in K_2 \}
	\end{equation*}
	then $\Phi(K_\epsilon)/\epsilon$ fulfills \eqref{eq:hp_blow_up_1_intr} and \eqref{eq:hp_blow_up_2_intr} with $K=K_1$.
 We remark that it is possible to generalize this idea and produce other examples.
\item 
 In the case of non-flat boundary, we have that conditions \eqref{eq:hp_blow_up_1_intr} and \eqref{eq:hp_blow_up_2_intr} hold
e.g. when $K_\epsilon$ is the image through $\Phi^{-1}$, for some $\Phi\in\mathcal{C}$, of sets like the ones in (i)--(ii). A remarkable case is when $\Phi(x',x_N)=(x',x_N-g(x'))$ in a neighbourhood of the origin, so that the restriction of $\Phi$ to $\partial\Omega$  is the orthogonal projection of $\partial\Omega$ onto its tangent hyperplane at $0$.
Hence assumption \eqref{eq:hp_blow_up_2_intr}  is satisfied, for example, if the each set  $K_\epsilon$ is a compact subset of $\partial \Omega$ whose orthogonal projection on the hyperplane tangent to $\partial\Omega$ at $0$ is of the form $\epsilon K$, for $K$ being a compact subset of $\R^{N-1}$.
\item 
Finally, it is easy to prove that assumptions \eqref{eq:hp_blow_up_1_intr} and \eqref{eq:hp_blow_up_2_intr} hold for
\[
K_\epsilon:=B_\epsilon\cap \partial \Omega, \quad\text{with } K=\R^{N-1}\cap\overline{B_1}.
\]
\end{enumerate}
\end{examples}

The last ingredient needed to detect the sharp asymptotics of the Sobolev capacity of shrinking  sets is the notion of vanishing order for the limit eigenfunction. Let 
\[
L^2(\mathbb{S}^{N-1}_+):=\left\{\psi\colon \mathbb{S}^{N-1}_+\to \R:\text{ $\psi$ is measurable and 
}\int_{\mathbb{S}^{N-1}_+}\abs{\psi}^2\ds<\infty  \right\},
\]
where $\mathbb{S}^{N-1}_+:=\{(x_1,\dots,x_N)\in \R^N \colon \abs{x}=1,~x_N>0 \}$. Moreover let
\[
H^1(\mathbb{S}^{N-1}_+):=\{\psi\in L^2(\mathbb{S}^{N-1}_+)\colon \nabla_{\mathbb{S}^{N-1}}\psi\in L^2(\mathbb{S}^{N-1}_+)  \}.
\]
The following proposition asserts that the limit eigenfunction $\phi_0$ behaves like a harmonic polynomial near the origin.
\begin{proposition}\label{prop:vanish_phi_0}
Let $\Omega$ satisfy assumption \eqref{eq:hp_boundary_1} and let $\phi_0$ be as in \eqref{eq:phi0}. 
	Then there exists $\gamma\in\N$ (possibly $0$) and $\Psi\in C^\infty(\overline{\mathbb{S}^{N-1}_+})$, $\Psi\neq 0$ such that, for all $\Phi\in\mathcal{C}$, there holds
	\begin{equation}\label{eq:phi_0_psi_gamma}
		\frac{\phi_0(\Phi^{-1}(\epsilon x))}{\epsilon^\gamma}\to \abs{x}^{\gamma}\Psi\left( \frac{x}{\abs{x}} \right)\quad\text{in }H^1(B_R^+)\text{ as }\epsilon\to 0,
	\end{equation}
	for all $R>0$, where $B_R^+:=B_R\cap \R^N_+$.

Furthermore, for every $R>0$,
\begin{equation}\label{eq:1}
  \epsilon^{-N-2\gamma}\int_{\Omega\cap B_{R\epsilon}}\varphi_0^2(x)\dx\to\int_{B_R^+}\psi^2_\gamma(x)\dx
\quad\text{as }\epsilon\to 0
\end{equation}
and 
\begin{equation}\label{eq:2}
  \epsilon^{-N-2\gamma+2}\int_{\Omega\cap B_{R\epsilon}}|\nabla \varphi_0(x)|^2\dx\to\int_{B_R^+}|\nabla\psi_\gamma(x)|^2\dx
\quad\text{as }\epsilon\to 0
\end{equation}
where \begin{equation}\label{eq:psi_gamma}
	\psi_\gamma(x):=\abs{x}^{\gamma}\Psi\left( \frac{x}{\abs{x}} \right).
\end{equation}
\end{proposition}
We observe that the function $\Psi$ appearing in \eqref{eq:phi_0_psi_gamma} and \eqref{eq:psi_gamma} is necessarily a spherical harmonic of degree $\gamma$ which
is symmetric with respect to the equator $x_N=0$, hence
satisfying homogeneous Neumann boundary conditions on $\{x_N=0\}$. More precisely, $\Psi$ solves
\[
\begin{bvp}
-\Delta_{\mathbb{S}^{N-1}} \Psi &=\gamma(N+\gamma-2)\Psi, &&\text{in }\mathbb{S}^{N-1}_+,\\
\nabla_{\mathbb{S}^{N-1}}\Psi\cdot \bm{e}_N&=0, &&\text{on }\partial\mathbb{S}^{N-1}_+,
\end{bvp}
\]
where $\bm{e}_N=(0,\dots,0,1)$. The integer number $\gamma$ is called the \emph{vanishing order} of $\phi_0$ in the origin. We also mention \cite{Bers1955,Felli2013,Nirenberg1959,Robbiano1988} for asymptotic behaviour of solutions to elliptic PDEs.
\begin{remark}\label{rmk:Psi}
	We observe that the restriction of $\Psi$ to the $N-2$ dimensional unit sphere $\partial\mathbb{S}^{N-1}_+$ cannot vanish everywhere. Indeed this would mean that the nonzero harmonic function $\psi_\gamma$ defined in \eqref{eq:psi_gamma} vanishes on $\partial\R^N_+$ together with is normal derivative; 
	but then the trivial extension of $\psi_\gamma$ to the whole $\R^N$ would violate the classical unique continuation principle (see \cite{Wolff1992}), thus giving rise to a contradiction.
\end{remark}

 For $N\geq 3$, let us introduce the Beppo Levi space $\mathcal{D}^{1,2}(\overline{\R^N_+})$ defined as the completion of $C_c^\infty(\overline{\R^N_+})$ with respect to the norm
\[
\norm{u}_{\mathcal{D}^{1,2}(\overline{\R^N_+})}:=\left(\int_{\R^N_+}\abs{\nabla u}^2\dx \right)^{1/2}.
\]
Furthermore, for any compact $K\sub \overline{\R^N_+}$, we define the space $\mathcal{D}^{1,2}(\overline{\R^N_+}\setminus K)$ as the closure of $C_c^\infty(\overline{\R^N_+}\setminus K)$ in $\mathcal{D}^{1,2}(\overline{\R^N_+})$. Thereafter we introduce a notion of capacity that will appear in the asymptotic expansion of $\Capa_{\bar{\Omega},c}(K_\epsilon,\phi_0)$, when $\epsilon\to 0$. 
\begin{definition}\label{def:rel_capacity}
	 For any compact $K\sub {\overline{\R^N_+}}$ and for any $f\in\mathcal{D}^{1,2}({\overline{\R^N_+}})$ we define the \emph{relative $f$-capacity} of $K$ in ${\overline{\R^N_+}}$ as
	\[
	\mathrm{cap}_{\overline{\R^N_+}}(K,f):=\inf\left\{ \int_{\R^N_+}\abs{\nabla u}^2\dx\colon u\in \mathcal{D}^{1,2}(\overline{\R^N_+}),~u-f\in\mathcal{D}^{1,2}(\overline{\R^N_+}\setminus K) \right\}.
	\]
	If $f\in\mathcal{D}^{1,2}(\overline{\R^N_+})$ is equal to $1$ in a neighbourhood  of $K$, we denote by 
	\[
	\mathrm{cap}_{\overline{\R^N_+}}(K):=\mathrm{cap}_{\overline{\R^N_+}}(K,f)
	\]
	the relative capacity of $K$ in $\overline{\R^N_+}$.	The definition can be extended  to functions $f\in H^1_{\textup{loc}}(\overline{\R_+^N})$ by letting
	\[
	\mathrm{cap}_{\overline{\R^N_+}}(K,f):=\inf\left\{ \int_{\R^N_+}\abs{\nabla u}^2\dx\colon u\in \mathcal{D}^{1,2}(\overline{\R^N_+}),~u-\eta_K f\in\mathcal{D}^{1,2}(\overline{\R^N_+}\setminus K) \right\},
	\]
	where $\eta_K\in C_c^\infty(\overline{\R^N_+})$ is such that $\eta_K=1$ in a neighbourhood of $K$.
\end{definition}

\begin{remark}\label{rmk:capa_std}
  We remark that the relative Sobolev capacity  in $\overline{\R_+^N}$ of a compact set  $K\sub\partial{\R_+^N}$, here denoted by $\mathrm{cap}_{\overline{\R_+^N}}(K)$, actually coincides with half of the capacity of $K$, in the classical sense (see \cite[Chapter 2.1]{Maly1997}), defined as 
\[
\capa_{\R^N}(K)=\inf\left\{ \int_{\R^N}\abs{\nabla u}^2\dx\colon u\in \mathcal{D}^{1,2}(\R^N),~u-\eta_K \in\mathcal{D}^{1,2}(\R^N\setminus K) \right\}.
\]
Moreover we notice that $\mathrm{cap}_{\overline{\R_+^N}}(K)$ coincides, up to a constant, with the Gagliardo $\frac{1}{2}$-fractional capacity of $K$ in $\R^{N-1}$, see e.g. \cite{AFN} for the definition.
\end{remark}

In this framework we are able to state the second main result of our paper, which concerns the sharp behaviour of the function $\epsilon\mapsto \Capa_{\bar{\Omega},c}(K_\epsilon,\phi_0)$ as $\epsilon\to~\!\!0^+$.

\begin{theorem}\label{thm:blow_up} Let $N\geq3$.  Assume \eqref{eq:hp_boundary_1} holds true. Let $\left\{K_\epsilon\right\}_{\epsilon>0}\sub\overline{\Omega}$ be a family of compact sets concentrating at $\{0\}\sub\partial\Omega$ as $\epsilon\to 0$ and let \eqref{eq:hp_blow_up_1_intr}-\eqref{eq:hp_blow_up_2_intr} hold for some $\Phi\in\mathcal{C}$ and for some compact set $K\sub\overline{\R^N_+}$ satisfying $\mathrm{cap}_{\overline{\R^N_+}}(K)>0$. 
Let $\phi_0$ be as in \eqref{eq:phi0} and 
let $\gamma$, $\psi_\gamma$ be as in \eqref{eq:phi_0_psi_gamma}-\eqref{eq:psi_gamma}. Then
\[ 
\Capa_{\bar{\Omega},c}(K_\epsilon,\phi_0)=
\epsilon^{N+2\gamma-2}\big(\mathrm{cap}_{\overline{\R_+^N}}(K,\psi_\gamma)+o(1)\big),\quad\text{as }\epsilon\to 0, 
\]
with $\mathrm{cap}_{\overline{\R^N_+}}(K,\psi_\gamma)$ being as in Definition \ref{def:rel_capacity}.
\end{theorem}

Combining Theorems \ref{thm:blow_up} and \ref{thm:sharp_asymp} we directly obtain the following corollary.
\begin{corollary}\label{cor:asy}
  Under the same assumptions and with the same notations of both Theorems \ref{thm:blow_up} and \ref{thm:sharp_asymp}, we have that 
\[		\lambda_\epsilon-\lambda_0=\epsilon^{N+2\gamma-2}
\big(\mathrm{cap}_{\overline{\R_+^N}}(K,\psi_\gamma)+o(1)\big),\quad\text{as }\epsilon\to 0.
	\]
\end{corollary}

The expansion stated in Corollary \ref{cor:asy} provides the  sharp asymptotics of the eigenvalue variation 
if $\mathrm{cap}_{\overline{\R^N_+}}(K,\psi_\gamma)>0$. This happens e.g. whenever 
$K\sub\partial\R^N_+$ is a compact set such that $\mathrm{cap}_{\overline{\R^N_+}}(K)>0$, as proved in Proposition \ref{prop:capa_pos_1};
we observe that the validity of such result strongly relies on  the position of the nodal set of $\psi_\gamma$ with respect to the set $K$.

 On the other hand, if $K\sub\partial\R^N_+$ is compact, we have that $\mathrm{cap}_{\overline{\R^N_+}}(K)>0$ if its  $N-1$ dimensional Lebesgue measure is nonzero, see Proposition \ref{prop:measure_capa_1}.

\subsection{Sets scaling to an interior point}

Although the present study was mainly motivated by our interest in
the eigenvalue asymptotics for moving mixed Dirichlet-Neumann boundary conditions, our techniques also apply to another class of perturbations, without any substantial difference, in view of the various possibilities embraced by Theorem \ref{thm:sharp_asymp}. In particular, it is possible to state a result analogous to Theorem \ref{thm:blow_up} in the case in which the perturbing sets $K_\epsilon$ are concentrating at a point that lies in the interior of $\Omega$. 
In this case the limit problem is the one with homogeneous Neumann boundary conditions on $\partial\Omega$ and the perturbed problem can be thought of as $\Omega$ without a ``small'' hole, on which zero Dirichlet boundary conditions are prescribed.
We assume that $0\in\Omega$ is the ``limit'' of the concentrating subsets $K_\epsilon$ and we ask assumptions similar to \eqref{eq:hp_blow_up_1_intr}-\eqref{eq:hp_blow_up_2_intr} to be satisfied, that is
\begin{gather}
\text{there exists }M\sub\R^N~\text{compact such that }  K_\epsilon/\epsilon\sub M\quad\text{for all }\epsilon\in (0,1), \label{eq:hp_blow_up_1_omega_intr} \\
\begin{gathered}
\text{there exists }K\sub\R^N~\text{compact such that }\\
\R^N\setminus (K_\epsilon/\epsilon)\to \R^N\setminus K\quad\text{in the sense of Mosco, as }\epsilon\to 0.
\end{gathered}
\label{eq:hp_blow_up_2_omega_intr}
\end{gather}
 As before, these assumptions are fulfilled, for instance, in the case $K_\epsilon:=\epsilon K$, for a certain compact $K\sub\R^N$ such that $K_\epsilon\sub\Omega$ for every $\epsilon\in(0,1)$. Since $0\in\Omega$ is an interior point, from classical regularity results for elliptic equations (see e.g. \cite{Robbiano1988}), there exist $\kappa\in\N$ and a spherical harmonic $Z$ of degree $\kappa$ such that
\[
-\Delta_{\mathbb{S}^{N-1}}Z=\kappa(N+\kappa-2)Z\quad\text{in }\mathbb{S}^{N-1}
\]
and
\begin{equation}\label{eq:psi_hat_gamma}
\frac{\phi_0(\epsilon x)}{\epsilon^{\kappa}}\to 
\zeta_\kappa(x):=\abs{x}^{\kappa}Z\left(\frac{x}{\abs{x}}\right)\quad\text{in }H^1(B_R)~\text{as }\epsilon\to 0,
\end{equation}
for all $R>0$. We can now state the last main result of our paper, which is analogous to Theorem \ref{thm:blow_up}.

\begin{theorem}\label{thm:blow_up_3}
Let $N\geq3$
	and $\{K_\epsilon\}_{\epsilon>0} \sub\Omega$ be a family of compact sets concentrating at $\{0\}\sub\Omega$ as $\epsilon\to 0$. Let \eqref{eq:hp_blow_up_1_omega_intr}-\eqref{eq:hp_blow_up_2_omega_intr} hold 
  for some compact set $K\sub{\R^N}$ satisfying
$\mathrm{cap}_{\R^N}(K)>0$. 
Let 
 $\phi_0$ be as in \eqref{eq:phi0} and  $\kappa$, $\zeta_\kappa$ be as in \eqref{eq:psi_hat_gamma}.
Then
	\[	\Capa_{\bar{\Omega},c}(K_\epsilon,\phi_0)=\epsilon^{N+2\kappa-2}
\big(\mathrm{cap}_{\R^N}(K,\zeta_\kappa)+o(1)\big),\quad\text{as }\epsilon\to 0,
	\]
	where and $\mathrm{cap}_{\R^N}(K,\zeta_\kappa)$ is the standard Newtonian $\zeta_\kappa$-capacity of $K$ (see Definition \ref{def:newt_capa}).
\end{theorem}
We point out that, in general, $\mathrm{cap}_{\R^N}(K,\zeta_\kappa)$ may not be strictly positive, since $K$, still having positive capacity, may happen to be a subset of the zero level set of $\zeta_\kappa$. In Lemma \ref{lemma:capa_pos} we provide sufficient conditions for 
$\mathrm{cap}_{\R^N}(K,\zeta_\kappa)$ to be strictly positive:
e.g. this happens when $K$ has nonzero capacity whereas the intersection of $K$ with the nodal set of $\zeta_\kappa$ has zero capacity.
We refer to \cite[Theorem 4.15]{Evans2015} for a sufficient condition for $\capa_{\R^N}(K)>0$: more precisely, we have that $\capa_{\R^N}(K)>0$ if its $N$-dimensional Lebesgue measure is nonzero.

The paper is organized as follows: in Section \ref{sec:prelim} we focus on the notion of capacity (as given in Definitions \ref{def:capacity} and \ref{def:rel_capacity}) and we prove some important properties (such as Propositions \ref{prop:capa_pos_1} and \ref{prop:measure_capa_1}), also in relation with the notion of concentration of sets. In Section \ref{sec:continuity} we prove the continuity of eigenvalues $\lambda_n(\Omega;K)$ with respect to $\Capa_{\bar{\Omega}}(K)$, i.e. Theorem \ref{thm:cont_eigen}. In Section \ref{sec:sharp_asym} we prove our first main result Theorem \ref{thm:sharp_asymp}. In Section \ref{sec:blow_up} we prove our second main result Theorem \ref{thm:blow_up} and finally, in Section \ref{sec:blow_up_general}, we prove Theorem \ref{thm:blow_up_3}.

\subsection{Notation}
Let us fix some notation we use throughout the paper:
\begin{itemize}
	\item[-] $\N_*:=\N\setminus\{0\}$;
	\item[-] $B_R:=\{x\in\R^N\colon \abs{x}< R \}$ and $S_R=\partial B_R$ for, respectively, balls and spheres centered at the origin;
	\item[-] $\R^N_+:=\{(x_1,\dots,x_N)\in\R^N\colon x_N>0\}$ for the upper half space;
	\item[-] we may identify $\R^{N-1}:=\partial\R^N_+$;
	\item[-] $B_R^+:=B_R\cap\R^N_+$ and $S_R^+:=\partial B_R^+\cap \R^N_+$ for half balls and half spheres;
		\item[-] $\mathbb{S}^{N-1}:=S_1$ and $\mathbb{S}^{N-1}_+:=S_1^+$ denote respectively the unitary sphere and upper unitary half-sphere;
	\item[-] $B_R':=B_R\cap\partial\R^N_+$.
\end{itemize}

\section{Preliminaries on concentration of sets and capacity}\label{sec:prelim}

In this section we focus on the notions of capacity and  concentration of sets (see Definitions \ref{def:capacity} and \ref{def:concentr}): we prove some basic properties and we investigate their mutual relations. We start by mentioning the existence of a capacitary potential, that is to say a function that achieves the infimum in the definition of capacity. We omit the proof since it follows the classical one.

\begin{proposition}[Capacity is Achieved]\label{prop:def_V_K_f}
Let $K\subseteq\overline{\Omega}$ be compact, $f\in H^1(\Omega)$ and $c\in L^\infty(\Omega)$ satisfying \eqref{eq:c_assumption}. The $f$-capacity of $K$, as introduced in Definition \ref{def:f-capacity}, is uniquely achieved, i.e. there exists a unique $V_{K,f,c}\in H^1(\Omega)$ which satisfies 
\begin{equation*}
V_{K,f,c}-f \in \HO{\Omega}{K} \qquad\text{and}\qquad
\Capa_{\bar{\Omega},c}(K,f)=q(V_{K,f,c}).
\end{equation*}
\end{proposition}
Since in the following the function $c$ is fixed, for the sake of brevity we will  write  
\[
V_{K,f}:=V_{K,f,c}
\]
 omitting the dependence on $c$ in the  notation.
 We observe that $V_{K,f}$ satisfies
	\begin{equation*}
	\begin{bvp}
	-\Delta V_{K,f}+cV_{K,f} &=0, &&\text{in }\Omega\setminus K, \\
	\frac{\partial V_{K,f}}{\partial \nnu}&=0, &&\text{on } \partial\Omega\setminus K,\\
	V_{K,f}&=f , &&\text{on } K,
	\end{bvp}
	\end{equation*}
in a weak sense, that is $V_{K,f}-f\in\HO{\Omega}{K}$ and
	\begin{equation}\label{eq:weak_V_K_f}
		q(V_{K,f},\phi)=\int_{\Omega}(\nabla V_{K,f}\cdot \nabla \phi+cV_{K,f} \phi)\dx=0\quad\text{for all }\phi\in \HO{\Omega}{K}.
	\end{equation}

\begin{remark}\label{rmk:def_cap}
	In the particular case $c,f\equiv 1$ (as in Definition \ref{def:capacity}), we have that 
the potential $V_K:=V_{K,1} \in H^1(\Omega)$ satisfies
	\begin{equation}\label{eq:weak_V_K_1}
				V_K-1\in\HO{\Omega}{K} \qquad\text{and}\qquad
				\int_{\Omega}(\nabla V_K\cdot\nabla \phi+V_K\phi)\dx=0\quad\text{for all }\phi\in \HO{\Omega}{K}.
	\end{equation}
It is easy to verify that $V_K^-,(V_K-1)^+\in \HO{\Omega}{K}$, so that we can choose $\phi=V_K^-$ and $\varphi=(V_K-1)^+$ as test functions in the above equation, thus obtaining that $V_K^-\equiv 0$ and $(V_K-1)^+\equiv 0$, i.e. 
\begin{equation}\label{eq:bound-VK}
  0\leq V_K(x)\leq 1\quad\text{for a.e. }x\in \Omega.
\end{equation}
\end{remark}

The following proposition asserts that the Sobolev spaces $H^1(\Omega)$ and $\HO{\Omega}{K}$ coincide if and only if the set $K$ has zero capacity, and draws conclusions on the eigenvalues of \eqref{eq:weak_mixed}.

\begin{proposition}\label{prop:Kequiv}
	Let $K\subseteq\overline{\Omega}$ be compact. The following three assertions are equivalent:
	\begin{itemize}
		\item[(i)] $\Capa_{\bar{\Omega}}(K)=0$;
		\item[(ii)] $H^1(\Omega)=\HO{\Omega}{K}$;
		\item[(iii)] $\lambda_n(\Omega;K)=\lambda_n$ for every $n\in\N_*$.
	\end{itemize}
\end{proposition}
\begin{proof}
	In order to prove that (i) implies (ii) it is sufficient to prove that $H^1(\Omega)\subseteq \HO{\Omega}{K}$ since the converse is trivial. We actually prove that $C^\infty(\overline{\Omega})\subseteq \HO{\Omega}{K}$ and the claim follows by density.
	By assumption (i), there exists $\{u_n\}_{n\geq1}\subset H^1(\Omega)$ such that $u_n-1\in \HO{\Omega}{K}$ for every $n\in\N_*$ and $\norm{u_n}_{H^1(\Omega)}\to0$ as $n\to\infty$.	
	Let $u\in C^\infty(\overline{\Omega})$ and let us consider the sequence $\{u(1-u_n)\}_{n\geq1}\subset \HO{\Omega}{K}$. We  claim that  $u(1-u_n)\to u$ in $ H^1(\Omega)$. Indeed
	\begin{align*}
	\norm{u-u(1-u_n)}_{H^1(\Omega)}^2 &=\norm{u u_n}_{H^1(\Omega)}^2 \\
	&\leq 2\int_{\Omega}(u^2\abs{\nabla u_n}^2+u_n^2\abs{\nabla u}^2)\dx+\int_{\Omega}u^2 u_n^2\dx \\
	&\leq 4 \max\{\norm{u}_{L^\infty(\Omega)}^2,\norm{\nabla u}_{L^\infty(\Omega)}^2\}\norm{u_n}_{H^1(\Omega)}^2\to 0,
	\end{align*}
	as $n\to\infty$.
	
	We now prove that (ii) implies (i). Let us consider the equation \eqref{eq:weak_V_K_1} solved by $V_K$.
	Since $\HO{\Omega}{K}=H^1(\Omega)$, we can choose $\varphi=V_K$ in \eqref{eq:weak_V_K_1} and then reach the conclusion.
	
	Finally, let us show that (ii) is equivalent to (iii). The fact that (ii) implies (iii) follows from the min-max characterization \eqref{eq:minmax}. Conversely, suppose that (iii) holds, i.e. $\lambda_n(\Omega;K)=\lambda_n$ for every $n\in\N_*$. Then for every $n\in\N_*$ there exists an eigenfunction belonging to $H^1_{0,K}(\Omega)$ associated to $\lambda_n$. By the Spectral Theorem, there exists an orthonormal basis of $H^1(\Omega)$ made of $H^1_{0,K}$-functions, which implies that property (ii) holds.
\end{proof}

\begin{remark}\label{rmk:zero_cap_equiv}
	An inspection of the proof of Proposition \ref{prop:Kequiv} shows that (ii) actually implies that $\Capa_{\bar{\Omega},c}(K,f)=0$ for all $f\in H^1(\Omega)$, and so
	\[
	\Capa_{\bar{\Omega}}(K)=0\quad\text{if and only if} \quad \Capa_{\bar{\Omega},c}(K,f)=0\quad\text{for all }f\in H^1(\Omega).
	\]
	Moreover, for any $f\in H^1(\Omega)$,
we
 trivially have that $\Capa_{\bar{\Omega},c}(K,f)=0$ if and only if  $V_{K,f}=0$.
\end{remark}

\begin{example}[Capacity of a Point]\label{ex:capa_point}
	Let $x_0\in\overline{\Omega}$, then $\Capa_{\bar{\Omega}}(\{x_0\})=0$.
\end{example}
\begin{proof}
	If $N\geq3$, let $v_n\in C^\infty(\overline{\Omega})$ be such that
	$v_n(x)=1$ if $x\in B(x_0,\frac1n)\cap \overline{\Omega}$, 
	$v_n(x)=0$ if $x\in \overline{\Omega}\setminus B(x_0,\frac2n)$, 
$0\leq v_n(x)\leq 1$ and $\abs{\nabla v_n(x)}\leq 2n$ for all $x\in \overline{\Omega}$.
	It is easy to prove that $q(v_n)\to 0$, as $n\to \infty$ thus concluding the proof for $N\geq 3$.
If $N=2$, we can instead consider
$v_n\in H^1(\Omega)$  defined as 
	$v_n(x)=1$ if $x\in B(x_0,\frac1n)\cap \overline{\Omega}$, 
	$v_n(x)=0$ if $x\in \overline{\Omega}\setminus B(x_0,\frac1{\sqrt{n}})$, 
$v_n(x)=(\log n)^{-1}(-\log n-2\log |x-x_0|)$ 
if $x\in \overline{\Omega}\cap \big(B(x_0,\frac1{\sqrt{n}})\setminus B(x_0,\frac1{n})\big)$.
	It is easy to prove that $q(v_n)\to 0$, as $n\to \infty$ thus concluding the proof for $N=2$.
\end{proof}

\begin{remark}\label{lemma:concentr_propert}
	Let $\{K_\epsilon\}_{\epsilon>0},K\sub\overline{\Omega}$ be compact sets such that $K_\epsilon$ is concentrating at $K$ as $\epsilon\to 0$. Then, for any $\phi\in C_c^\infty(\overline{\Omega}\setminus K)$, there exists $\epsilon_\phi>0$ such that
$\phi\in C_c^\infty(\overline{\Omega}\setminus K_\epsilon)$ for all $\epsilon<\epsilon_\phi$.
Moreover $\underset{\epsilon>0}{\bigcap}K_\epsilon\subseteq K$.
      \end{remark}

\begin{example}\label{ex:concentr_sets}
	An example of concentrating sets is a family of compact sets decreasing as $\epsilon\to0$. Precisely, let $\{K_\epsilon\}_{\epsilon>0}$ be a family of compact subsets of $\overline{\Omega}$ such that $K_{\epsilon_2}\subseteq K_{\epsilon_1}$ for any $\epsilon_2\leq\epsilon_1$ and let $K\subseteq \overline{\Omega}$ be a compact set such that $K=\cap_{\epsilon>0}K_\epsilon$. Then, arguing by contradiction, thanks to Bolzano-Weierstrass Theorem in $\R^N$, it is easy to prove that $K_\epsilon$ is concentrating at $K$.
\end{example}

With the next proposition we emphasize what is the relation between the notion of concentration of sets and that of convergence of capacities: it turns out that convergence holds if the limit set has zero capacity.

\begin{proposition}\label{prop:cont_cap}
	Let $K\subseteq\overline{\Omega}$ be a compact set and let $\{K_\epsilon\}_{\epsilon>0}$ be a family of compact subsets of $\overline{\Omega}$ concentrating at $K$.
If $\Capa_{\bar{\Omega}}(K)=0$ then 
	\begin{equation*}
		V_{K_\epsilon,f}\to V_{K,f}\quad\text{in }H^1(\Omega)\quad\text{and}\quad 
		\Capa_{\bar{\Omega},c}(K_\epsilon,f)\to \Capa_{\bar{\Omega},c}(K,f)\quad\text{as }\epsilon\to 0
              \end{equation*}
for all $f\in H^1(\Omega)$ and all $c\in L^\infty(\Omega)$ satisfying \eqref{eq:c_assumption}. This result holds true, in particular, for the Sobolev capacity (see Definition \ref{def:capacity}) and its potentials, corresponding to the case in which $c\equiv 1$.
\end{proposition}
\begin{proof}
	Since $q(V_{K_\epsilon,f})\leq q(f)$ for all $\epsilon>0$, then $\{V_{K_\epsilon,f}\}_\epsilon$ is bounded in $H^1(\Omega)$ and so there exists $W\in H^1(\Omega)$ such that, along a sequence $\epsilon_n\to0$,
	\begin{equation*}
		V_{K_{\epsilon_n},f}\rightharpoonup W\quad\text{weakly in }H^1(\Omega)~\text{as }n\to\infty,
	\end{equation*}
	that is
	\begin{equation}\label{eq:cont_cap_1}
		\int_{\Omega}(\nabla V_{K_{\epsilon_n},f}\cdot \nabla \phi+cV_{K_{\epsilon_n},f} \phi)\dx\to\int_{\Omega}(\nabla W\cdot \nabla \phi+cW \phi)\dx \quad\text{for all }\phi\in H^1(\Omega).
	\end{equation}
	Therefore, taking into account Remark \ref{lemma:concentr_propert} and the equation solved by $V_{K_\epsilon,f}$ \eqref{eq:weak_V_K_f}, we have that
	\begin{equation}\label{eq:cont_cap_2}
		\int_{\Omega}(\nabla W\cdot \nabla \phi+cW \phi)\dx=0 
	\end{equation}
for all $\phi\in C_c^\infty(\overline{\Omega}\setminus K)$
	and then, by density, for all $\phi\in \HO{\Omega}{K}$. 
Moreover, taking $\varphi=V_{K,f}-f$ (respectively $\varphi=V_{K_\epsilon,f}-f$) in the equation \eqref{eq:weak_V_K_f} for $V_{K,f}$ (respectively $V_{K_\epsilon,f}$), we obtain
	\begin{equation}\label{eq:cont_cap_4}
		\Capa_{\bar{\Omega},c}(K,f)=\int_{\Omega}(\nabla V_{K,f}\cdot\nabla f +cV_{K,f}f)\dx,
	\end{equation}
	respectively
	\begin{equation}\label{eq:cont_cap_3}
		\Capa_{\bar{\Omega},c}(K_\epsilon,f)=\int_{\Omega}(\nabla V_{K_\epsilon,f}\cdot\nabla f +cV_{K_\epsilon,f}f)\dx.
	\end{equation}
	From Proposition \ref{prop:Kequiv}, we have that $\HO{\Omega}{K}=H^1(\Omega)$ and then \eqref{eq:cont_cap_2} yields $W=V_{K,f}=0$; on the other hand, from \eqref{eq:cont_cap_3} and \eqref{eq:cont_cap_4} it follows that
	\[
		\Capa_{\bar{\Omega}}(K_{\epsilon_n},f)\to \Capa_{\bar{\Omega},c}(K,f)=0\quad\text{as }n\to\infty.
	\]
	Urysohn's Subsequence Principle concludes the proof.
\end{proof}

The following lemma is a fundamental step in the proof of our main results.
It states that, when a sequence of sets is concentrating, as $\epsilon\to0$, at a zero capacity set, the squared $L^2(\Omega)$-norm of the associated capacitary potentials is negligible, as $\epsilon\to0$, with respect to the capacity.

\begin{lemma}\label{lemma:L^2_norm}
	Let $K\subseteq\overline{\Omega}$ be compact and let $\{K_\epsilon\}_{\epsilon>0}$ be a family of compact subsets of $\overline{\Omega}$ concentrating at $K$. If $\Capa_{\bar{\Omega}}(K)=0$, then
	\[
		\int_{\Omega}\abs{V_{K_\epsilon,f}}^2\dx=o(\Capa_{\bar{\Omega},c}(K_\epsilon,f))\quad\text{as }\epsilon\to0
	\]
	for all $f\in H^1(\Omega)$ and all $c\in L^\infty(\Omega)$ satisfying \eqref{eq:c_assumption}.
\end{lemma}
\begin{proof}
	Assume by contradiction that, for a certain $f\in H^1(\Omega)$, there exists $\epsilon_n\to 0$ and $C>0$ such that
	\[
		\int_{\Omega}\abs{V_{K_{\epsilon_n},f}}^2\dx\geq \frac{1}{C}\Capa_{\bar{\Omega},c}(K_{\epsilon_n},f).
	\]
	Let
	\[
		W_n:=\frac{V_{K_{\epsilon_n},f}}{\norm{V_{K_{\epsilon_n},f}}}_{L^2(\Omega)}.
	\]
	Then $\norm{W_n}_{L^2(\Omega)}=1$ and
	\[
		\norm{\nabla W_n}_{L^2(\Omega)}^2
+ \int_\Omega cW_n^2\dx
=\frac{\Capa_{\bar{\Omega},c}(K_{\epsilon_n},f)}{\norm{V_{K_{\epsilon_n},f}}_{L^2(\Omega)}^2}\leq C.
	\]
	Hence $\{W_n\}_n$ is bounded in $H^1(\Omega)$ and so there exists $W\in H^1(\Omega)$ such that $W_n\rightharpoonup W$ weakly in $H^1(\Omega)$, up to a subsequence, as $n\to\infty$. By compactness of the embedding $H^1(\Omega)\hookrightarrow L^2(\Omega)$ we have that $\norm{W}_{L^2(\Omega)}=1$. 
	Using Remark \ref{lemma:concentr_propert}, we can pass to the limit in the equation satisfied by $W_n$ and then obtain
	\[
		\int_{\Omega}(\nabla W\cdot\nabla \phi+cW\phi)\dx=0\quad\text{for all }\phi\in C_c^\infty(\overline{\Omega}\setminus K).
	\]
	On the other hand, since $\Capa_{\bar{\Omega}}(K)=0$, in view of 
Proposition \ref{prop:Kequiv} $C_c^\infty(\overline{\Omega}\setminus K)$ is dense in $H^1(\Omega)$ and so $W=0$, thus a contradiction arises.
\end{proof}

We are now going to prove that the term $\mathrm{cap}_{\overline{\R^N_+}}(K,\psi_\gamma)$, appearing in the expansion stated in Corollary \ref{cor:asy}, 
is nonzero whenever 
$K\sub\partial\R^N_+$ is a compact set such that $\mathrm{cap}_{\overline{\R^N_+}}(K)>0$; to this aim we prove
a more general lemma concerning the standard (Newtonian) capacity of a set, whose definition we recall below. For any open set $U\sub\R^N$, we denote by $\mathcal{D}^{1,2}(U)$ the completion of $C_c^\infty(U)$ with respect to the $L^2(U)$-norm of the gradient.

\begin{definition}\label{def:newt_capa}
	For $N\geq3$,  let $K\sub\R^N$  be a compact set and let $\eta_K\in C_c^\infty(\R^N)$ be such that $\eta_K=1$ in a neighbourhood of $K$.
If $f\in H^1_{\rm loc}(\R^N)$,  the following quantity 
	\[
		\capa_{\R^N}(K,f)=\inf\left\{\int_{\R^N}\abs{\nabla u}^2\dx: u\in \mathcal{D}^{1,2}(\R^N),~u-f\eta_K\in\mathcal{D}^{1,2}(\R^N\setminus K)  \right\}
	\]
is called the \emph{$f$-capacity} of $K$.
	For $f=1$,  $\capa_{\R^N}(K):=\capa_{\R^N}(K,1)$ is called the \emph{capacity} of $K$ (as already introduced in Remark \ref{rmk:capa_std}).
\end{definition}

\begin{lemma}\label{lemma:capa_pos}
	Let $N\geq3$ and $K\sub \R^N$ be a compact set such that $\mathrm{cap}_{\R^N}(K)>0$. Let 
$f\in~\!\!C^\infty(\R^N)$ and let $Z_f:=\{x\in\R^N\colon f(x)=0 \}$. If $\mathrm{cap}_{\R^N}(Z_f\cap K)<\mathrm{cap}_{\R^N}(K)$, then 
	\[
	\mathrm{cap}_{\R^N}(K,f)>0.
	\]
\end{lemma}
\begin{proof}
	In this proof we make use of some properties of the classical Newtonian capacity of a set and we refer to \cite[Chapter 2]{Maly1997} for the details. Let us consider a sequence of  bounded open sets $\mathcal{U}_n\sub\R^N$ such that
	\begin{equation*}
		 Z_f\cap K\sub \mathcal{U}_{n+1}\sub \mathcal{U}_n,\quad\text{for all }n\text{ and}\quad
		\bigcap_{n\geq 1}\mathcal{\overline{U}}_n=Z_f\cap K.
	\end{equation*}
	Let $K_n:=K\setminus \mathcal{U}_n$. 
Since $K\subseteq K_n\cup \overline {\mathcal U_n}$, 
by subaddittivity and monotonicity of the capacity
	\[
		\capa_{\R^N}(K_n)\geq  \capa_{\R^N}(K)-\capa_{\R^N}(\mathcal{\overline{U}}_n).
	\]
	Moreover, since $\cap_{n\geq 1}\mathcal{\overline{U}}_n=Z_f\cap K$, then $\capa_{\R^N}(Z_f\cap K)=\lim_{n\to\infty}\capa_{\R^N} (\mathcal{\overline{U}}_n)$, and so \begin{equation}\label{eq:capa_f_pos_1}
		\capa_{\R^N}(K_n)>0
	\end{equation}
	for large $n$, by the assumption  $\capa_{\R^N}(K)-\capa_{\R^N}(Z_f\cap K)>0$. Now we claim that 
	\begin{equation}\label{eq:capa_f_pos_2}
		\capa_{\R^N}(K,|f|)>0.
	\end{equation}
	Let us fix a sufficiently large $n$ in order for \eqref{eq:capa_f_pos_1} to hold and let us set
	\begin{equation*}
		c_n:=\frac{1}{2}\inf_{K_n}|f|=\frac{1}{2}\min_{K_n}|f|>0.
	\end{equation*}
	By definition of $K_n$ we have $|f|\geq 2c_n>c_n$ on $K_n$ and therefore, by continuity, $|f|>c_n$ in an open neighbourhood of $K_n$. Let $\eta_K\in C_c^\infty(\R^N)$ be such that $\eta_K=1$ in a neighbourhood of $K$ and let $u_n\in \mathcal{D}^{1,2}(\R^N)$ be an arbitrary function such that $u_n-\eta_K|f|\in \mathcal{D}^{1,2}(\R^N\setminus K_n)$. We define
	\[
		v_n:=\min\{1,u_n/c_n \}\in\mathcal{D}^{1,2}(\R^N).
	\]
	We have that $v_n-\eta_{K_n}\in\mathcal{D}^{1,2}(\R^N\setminus K_n)$, where $\eta_{K_n}\in C_c^\infty(\R^N)$ is equal to $1$ in a neighbourhood of $K_n$. Therefore $v_n$ is an admissible competitor for $\capa_{\R^N}(K_n)$ and also, by truncation, the energy of $v_n$ is lower than the energy of $u_n/c_n$. Hence
	\begin{equation*}
		\capa_{\R^N}(K_n)\leq \int_{\R^N}\abs{\nabla v_n}^2\dx\leq \int_{\R^N}\frac{\abs{\nabla u_n}}{c_n^2}\dx.
	\end{equation*}
	By arbitrariness of $u_n$, we have that $\capa_{\R^N}(K_n,|f|)\geq c_n^2\capa_{\R^N}(K_n)>0$. Moreover, by monotonicity, $\capa_{\R^N}(K,|f|)\geq \capa_{\R^N}(K_n,|f|)>0$ and so \eqref{eq:capa_f_pos_2} is proved. Finally, we claim that 
	\begin{equation}\label{eq:capa_f_pos_3}
		\capa_{\R^N}(K,f)\geq\capa_{\R^N}(K,|f|).
	\end{equation}
	Indeed, if $\xi\in\mathcal{D}^{1,2}(\R^N)$ is such that $\xi-\eta_K f\in\mathcal{D}^{1,2}(\R^N\setminus K)$, then $\abs{\xi}-\eta_K|f|\in\mathcal{D}^{1,2}(\R^N\setminus K)$. Hence
	\[
		\capa_{\R^N}(K,|f|)\leq \int_{\R^N}|\nabla |\xi||^2\dx=\int_{\R^N}|\nabla\xi|^2\dx
	\]
	for all $\xi\in \mathcal{D}^{1,2}(\R^N)$ such that $\xi-\eta_K f\in\mathcal{D}^{1,2}(\R^N\setminus K)$, which implies \eqref{eq:capa_f_pos_3}. Combining \eqref{eq:capa_f_pos_2} and \eqref{eq:capa_f_pos_3} we can conclude the proof.	
\end{proof}

As an application of the previous lemma, we  obtain the following result.

\begin{proposition}\label{prop:capa_pos_1}
	For $N\geq3$, let $K\sub\partial\R^N_+$ be a compact set such that $\mathrm{cap}_{\overline{\R^N_+}}(K)>0$ and let $\psi_\gamma$ be as in \eqref{eq:psi_gamma}. Then
	\[
	\mathrm{cap}_{\overline{\R^N_+}}(K,\psi_\gamma)>0.
	\]
\end{proposition}

\begin{proof}
	Let $Z_{\psi_\gamma}=\{x\in\R^N\colon \psi_\gamma(x)=0 \}$ as in the statement of Lemma \ref{lemma:capa_pos}.
	We notice that $\capa_{\R^N}(K)=2 \mathrm{cap}_{\overline{\R}^N_+}(K)>0$ (see Remark \ref{rmk:capa_std}), so that the first assumption of Lemma \ref{lemma:capa_pos} holds.
	Concerning the second assumption, we have that $\capa_{\R^N}(Z_{\psi_\gamma}\cap K)=2 \mathrm{cap}_{\overline{\R^N_+}}(Z_{\psi_\gamma}\cap K)=0$, since the set $ Z_{\psi_\gamma}\cap K$ is $(N-2)$-dimensional, in view of Remark \ref{rmk:Psi}, and $(N-2)$-dimensional sets have zero capacity in $\R^N$, see e.g. \cite[Theorem 2.52]{Maly1997}. Then Lemma \ref{lemma:capa_pos} provides $\capa_{\R^N}(K,\psi_\gamma)>0$ and the proof follows by applying again Remark \ref{rmk:capa_std}.
\end{proof}

We conclude this section with the following lower bound of 
$\mathrm{cap}_{\overline{\R^N_+}}(K)$ in terms of its $N-1$ dimensional Lebesgue measure.

\begin{proposition}\label{prop:measure_capa_1}
	Let $N\geq3$ and $K\sub\partial\R^N_+$ be compact. Then there exists a constant $C>0$ (only depending on $N$) such that
	\[
	(|K|_{N-1})^{\frac{N-2}{N-1}}\leq C\, \mathrm{cap}_{\overline{\R^N_+}}(K),
	\]
	where $|\cdot|_{N-1}$ denotes the $N-1$ dimensional Lebesgue measure.
\end{proposition}

\begin{proof}
	By definition of $\mathcal{D}^{1,2}(\overline{\R^N_+})$ and $\mathrm{cap}_{\overline{\R^N_+}}(K)$, for every $\epsilon>0$ there exists $u\in C_c^\infty(\overline{\R^N_+})$ such that $u=1$ in an open neighbourhood $U$ of $K$ and
	\[
		\int_{\R^N_+}\abs{\nabla u}^2\dx\leq \mathrm{cap}_{\overline{\R^N_+}}(K)+\epsilon.
	\]
	On the other hand
	\[
		|K|_{N-1}\leq |U|_{N-1}=\int_U\abs{u}^{\frac{2(N-1)}{N-2}}\ds\leq \int_{\R^{N-1}}\abs{u}^{\frac{2(N-1)}{N-2}}\ds.
	\]
	By combining the two previous inequalities with the embedding $\mathcal{D}^{1,2}(\overline{\R^N_+})\hookrightarrow L^{\frac{2(N-1)}{N-2}}(\R^{N-1})$ we can conclude the proof.
\end{proof}

\section{Continuity of the eigenvalues with respect to the capacity}\label{sec:continuity}

5The aim of this section is to prove continuity of the eigenvalues $\lambda_n(\Omega;K)$, in the limit as $\Capa_{\bar{\Omega}}(K)\to 0$.

\begin{proof}[Proof of Theorem \ref{thm:cont_eigen}]
If $\Capa_{\bar{\Omega}}(K)=0$ the conclusion follows obviously from Proposition \ref{prop:Kequiv}. 
Let us assume that $\Capa_{\bar{\Omega}}(K)>0$. Then, by definition of $\Capa_{\bar{\Omega}}(K)$, there exists $v\in 
C^\infty(\overline{\Omega})$ such that $v-1\in C^\infty_{\rm c}(\overline{\Omega}\setminus K)$ 
and $\|v\|_{H^1(\Omega)}^2\leq 2 \Capa_{\bar{\Omega}}(K)$. Letting $w=(1-(1-v)^+)^+$, we have that $w\in W^{1,\infty}(\Omega)$, $0\leq w\leq1$ a.e. in $\Omega$, $w-1\in \HO{\Omega}{K}$, and $\|w\|_{H^1(\Omega)}^2\leq \|v\|_{H^1(\Omega)}^2\leq 2 \Capa_{\bar{\Omega}}(K)$.

Let $\phi_1,\dots,\phi_n$ be the eigenfunctions corresponding to $\lambda_1,\dots,\lambda_n$ and let 
$\Phi_i:=\phi_i(1-w)$, $i=1,\dots,n$. It's easy to prove that $\Phi_i\in \HO{\Omega}{K}$ for all $i=1,\dots,n$. Let us consider the  linear subspace of $\HO{\Omega}{K}$
\[
E_n:=\Span\{\Phi_1,\dots,\Phi_n\}.
\]
We claim that,  if $\Capa_{\bar{\Omega}}(K)$ is sufficiently small, $\{\Phi_i\}_{i=1}^n$ is linearly independent, thus implying that $\dim E_n=n$.
In order to compute $q(\Phi_i,\Phi_j)$, we test the equation satisfied by $\phi_i$ with $\phi_j(1-w)^2$. It follows that 
\[
	\int_{\Omega}[(1-w)^2\nabla\phi_i\cdot\nabla\phi_j +c(1-w)^2\phi_i\phi_j]\dx 
	=\int_{\Omega}[\lambda_i(1-w)^2\phi_i\phi_j+2(1-w)\phi_j\nabla\phi_i\cdot\nabla w]\dx.
\]
	Thanks to the previous identity, we are able to compute
\[
	q(\Phi_i,\Phi_j)=\int_{\Omega}[\phi_j(1-w)\nabla\phi_i\cdot\nabla w-\phi_i(1-w)\nabla\phi_j\cdot\nabla w
	+\phi_i\phi_j\abs{\nabla w}^2+\lambda_i\phi_i\phi_j(1-w)^2 ]\dx.
\]
From classical elliptic regularity theory (see e.g. \cite[Proposition 5.3]{stampacchia}) it is well-known that $\phi_i\in L^\infty(\Omega)$. Then, thanks also to H\"older inequality and \eqref{eq:c_assumption}, we have that	
	\[
	\abs{q(\Phi_i,\Phi_j)-\delta_{ij}\lambda_i}\leq C_1[(\Capa_{\bar{\Omega}}(K))^{1/2}+\Capa_{\bar{\Omega}}(K)],
	\]
	for a certain $C_1>0$ (depending only on $\|\varphi_i\|_{L^\infty(\Omega)}$ and $\lambda_i$, $i=1,\dots,n$), where $\delta_{ij}$ is the \emph{Kronecker's Delta}. The above inequality implies that 
	\[
		q(\Phi_i,\Phi_j)=\delta_{ij}\lambda_i+O((\Capa_{\bar{\Omega}}(K))^{1/2})\quad\text{as }\Capa_{\bar{\Omega}}(K)\to 0,
	\] 
hence there exists $\delta>0$ such that, if $\Capa_{\bar{\Omega}}(K)<\delta$, then $\Phi_1,\dots,\Phi_{n}$ are linearly independent. Let us now compute the $L^2$ scalar products
	\begin{equation*}
		\int_{\Omega}\Phi_i\Phi_j\dx=\int_{\Omega}\phi_i\phi_j(1-w)^2\dx 
		=\delta_{ij}-2\int_{\Omega}\phi_i\phi_j w\dx+\int_{\Omega} \phi_i\phi_j w^2\dx.
	\end{equation*}
	Arguing as before, by H\"older inequality we obtain that 	
	\begin{align*}
		\abs{\int_{\Omega}\Phi_i\Phi_j\dx -\delta_{ij}}
		&\leq 2\sqrt{2}\norm{\phi_i\phi_j}_{L^2(\Omega)}(\Capa_{\bar{\Omega}}(K))^{1/2}+2\norm{\phi_i\phi_j}_{L^\infty(\Omega)}\Capa_{\bar{\Omega}}(K) \\
		&\leq C_2[(\Capa_{\bar{\Omega}}(K))^{1/2}+\Capa_{\bar{\Omega}}(K)],
	\end{align*}
	for a certain $C_2>0$ (depending only on $\|\varphi_i\|_{L^\infty(\Omega)}$, $i=1,\dots,n$), i.e. 
	\[
		\int_{\Omega}\Phi_i\Phi_j\dx=\delta_{ij}+O((\Capa_{\bar{\Omega}}(K))^{1/2}) \quad\text{as }\Capa_{\bar{\Omega}}(K)\to 0.
	\]
	Now, from the min-max characterization \eqref{eq:minmax}, we have that
	\begin{align*}
		\lambda_n(\Omega;K)&\leq \max_{\substack{\alpha_1,\dots,\alpha_{n}\in\R \\ \sum_{i=1}^{n}\alpha_i^2=1}}\frac{ q\left(\sum_{i=1}^{n}\alpha_i \Phi_i \right)}{\sum_{i,j=1}^{n}\alpha_i\alpha_j\int_{\Omega}\Phi_i\Phi_j\dx} \\
		&=\max_{\substack{\alpha_1,\dots,\alpha_{n}\in\R \\ \sum_{i=1}^{n}\alpha_i^2=1}}\frac{\sum_{i,j=1}^{n}\alpha_i\alpha_j q(\Phi_i,\Phi_j)}{\sum_{i,j=1}^{n}\alpha_i\alpha_j(\delta_{ij}+O((\Capa_{\bar{\Omega}}(K))^{1/2}))} \\
		&= \max_{\substack{\alpha_1,\dots,\alpha_{n}\in\R \\ \sum_{i=1}^{n}\alpha_i^2=1}}\frac{ \sum_{i=1}^{n}\alpha_i^2\lambda_i +O((\Capa_{\bar{\Omega}}(K))^{1/2})}{1+O((\Capa_{\bar{\Omega}}(K))^{1/2})} \\
                                   &\leq \frac{\lambda_n+O((\Capa_{\bar{\Omega}}(K))^{1/2})}{1+O((\Capa_{\bar{\Omega}}(K))^{1/2})}=
\lambda_n+O((\Capa_{\bar{\Omega}}(K))^{1/2})
	\end{align*}
 as $\Capa_{\bar{\Omega}}(K)\to 0$.
\end{proof}

\section{Sharp asymptotics of perturbed eigenvalues}\label{sec:sharp_asym}

This section is devoted to the proof of Theorem \ref{thm:sharp_asymp}. To this aim, let us give a preliminary lemma concerning the inverse of the operator $-\Delta+c$, when it acts on functions that vanish on a compact set.

\begin{lemma}\label{lemma:A_K}
For $K\subseteq\overline{\Omega}$ compact, let  $A_K: H^1_{0,K}(\Omega)\to H^1_{0,K}(\Omega)$  be the linear bounded operator defined by
\begin{equation}\label{eq:A_K_def}
q(A_K (u),v)=(u,v)_{L^2(\Omega)} \quad\text{for every } u,v \in H^1_{0,K}(\Omega).
\end{equation}
Then
\begin{itemize}
		\item[(i)] $A_K$ is symmetric, non-negative and compact; in particular, $0$ belongs to its spectrum~$\sigma(A_K)$.
		\item[(ii)] $\sigma(A_K)\setminus\{0\}=\{\mu_n(\Omega;K)\}_{n\in\N_*}$ and $\mu_n(\Omega;K)=1/\lambda_n(\Omega;K)$ for every $n\in\N_*$.
		\item[(iii)] For every $\mu\in\R$ and $u\in H^1_{0,K}(\Omega)\setminus\{0\}$ it holds
		\begin{equation}\label{eq:A_K_spectral_theorem}
		\left(\dist(\mu,\sigma(A_K))\right)^2 \leq \frac{q(A_K(u)-\mu u)}{q(u)}.
		\end{equation}
\end{itemize}
\end{lemma}
\begin{proof}
(i) $A_K$ is clearly symmetric and non-negative; let us show that it is compact. We write $A_K=\mathcal{R}\circ\mathcal{I}$, where $\mathcal{I}: H^1_{0,K}(\Omega)\to (H^1_{0,K}(\Omega))^*$ is the compact immersion
\[
\phantom{a}_{(H^1_{0,K}(\Omega))^*}\langle \mathcal{I}(u),v\rangle_{H^1_{0,K}(\Omega)} = \int_\Omega uv \dx \quad \text{for every }u,v \in H^1_{0,K}(\Omega),
\]
and $\mathcal{R}:(H^1_{0,K}(\Omega))^* \to H^1_{0,K}(\Omega)$ is the Riesz isomorphism (on $H^1_{0,K}(\Omega)$ endowed with the scalar product $q$) given by
\[
q(\mathcal{R}(F),v)=
\phantom{a}_{(H^1_{0,K}(\Omega))^*}\langle F,v\rangle_{H^1_{0,K}(\Omega)} 
\]
for every $v \in H^1_{0,K}(\Omega)$ and $F\in (H^1_{0,K}(\Omega))^*$. Then $A_K$ is compact and $0\in \sigma(A_K)$ (see for example \cite[Theorem 6.16]{Helffer2013}).

\noindent (ii) Again by \cite[Theorems 6.16]{Helffer2013}, $\sigma(A_K)\setminus\{0\}$ consists of isolated eigenvalues having finite multiplicity. Being $q(\cdot)$ a norm over $H^1_{0,K}(\Omega)$, we have that $\mu\neq 0$ is an eigenvalue of $A_K$ if and only if there exists $u\in H^1_{0,K}(\Omega)$, $u\not\equiv0$, such that
\[
q(A_K(u),v)=\mu q(u,v) \quad\text{for every } v\in H^1_{0,K}(\Omega),
\]
so that $1/\mu=\lambda_n(\Omega;K)$ for some $n\in \N_*$.

\noindent (ii) This is a consequence of the Spectral Theorem (see for example
\cite[Theorem 6.21 and Proposition 8.20]{Helffer2013}).
\end{proof}

We have now all the ingredients to give the proof of Theorem \ref{thm:sharp_asymp}. It is inspired by \cite[Theorem 1.4]{AFHL} (see also \cite[Theorem 1.5]{AFN}).

\begin{proof}[Proof of Theorem \ref{thm:sharp_asymp}]
Let us recall that, by assumption, $\lambda_0=\lambda_{n_0}=\lambda_{n_0}(\Omega;\emptyset)$ is simple and that $\varphi_0$ is an associated $L^2(\Omega)$-normalized eigenfunction. Recall also that $\lambda_\epsilon=\lambda_{n_0}(\Omega;K_\epsilon)$.
	For simplicity of notation we write $V_\epsilon:=V_{K_\epsilon,\phi_0}$ and $C_\epsilon:=\Capa_{\bar{\Omega},c}(K_\epsilon,\phi_0)=q(V_\epsilon)$. Moreover we let $\psi_\epsilon:=\phi_0-V_\epsilon$, that is $\psi_\epsilon $ is the orthogonal projection of $\phi_0$ on $\HO{\Omega}{K_\epsilon}$ with respect to $q$. Indeed there holds
	\[
		q(\psi_\epsilon-\phi_0,\phi)=0\quad\text{for all }\phi\in\HO{\Omega}{K_\epsilon}.
	\]
We split the proof into three steps.

\smallskip\noindent{\bf Step 1.} We claim that 
	\begin{equation}\label{eq:claim1}
|\lambda_\epsilon-\lambda_0|=o(C_\epsilon^{1/2})\quad\text{as }\epsilon\to 0.
\end{equation}
	For any $\phi\in\HO{\Omega}{K_\epsilon}$, being $\lambda_0$ an eigenvalue of $q$, we have
	\begin{equation}\label{eq:psi_eps_phi}
		q(\psi_\epsilon,\phi)-\lambda_0(\psi_\epsilon,\phi)_{L^2(\Omega)}= q(\phi_0,\phi)-\lambda_0(\psi_\epsilon,\phi)_{L^2(\Omega)}=\lambda_0(V_\epsilon,\phi)_{L^2(\Omega)}.
	\end{equation}
According to the notation introduced in Lemma \ref{lemma:A_K}, we can rewrite \eqref{eq:psi_eps_phi}  as
	\begin{equation}\label{eq:psi_eps_phi2}
	(\psi_\epsilon,\phi)_{L^2(\Omega)} = \mu_0 q(\psi_\epsilon,\phi)-(V_\epsilon,\phi)_{L^2(\Omega)},
	\end{equation}
	where $\mu_0=\mu_{n_0}(\Omega;\emptyset)=\mu_{n_0}(\Omega;K)=1/\lambda_0$. By \eqref{eq:A_K_spectral_theorem} we have
	\begin{equation}\label{eq:dist_mu0_sigma}
	\big(\dist(\mu_0,\sigma(A_{K_\epsilon}))\big)^2 \leq \frac{q(A_{K_\epsilon}(\psi_\epsilon)-\mu_0\psi_\epsilon)}{q(\psi_\epsilon)}.
	\end{equation}
	From Proposition \ref{prop:cont_cap} it follows that
\[
|q(\varphi_0,V_\epsilon)|\leq \sqrt{q(\varphi_0)} \sqrt{q(V_\epsilon)} =
\sqrt{\lambda_0}\sqrt{C_\epsilon}=o(1)
\]
as $\epsilon\to0$, so that, using the definition of $\psi_\epsilon$, the denominator in the right hand side of \eqref{eq:dist_mu0_sigma} can be estimated as follows
\begin{equation}\label{eq:q_psi_esp}
q(\psi_\epsilon)=q(\varphi_0)+C_\epsilon-2q(\varphi_0,V_\epsilon)=\lambda_0+o(1)
\end{equation}
as $\epsilon\to0$. Concerning the numerator in the right hand side of \eqref{eq:dist_mu0_sigma}, the definition of $A_{K_\epsilon}$ and relation \eqref{eq:psi_eps_phi2} provide
\[
q(A_{K_\epsilon}(\psi_\epsilon),\phi)=(\psi_\epsilon,\phi)_{L^2(\Omega)}
=\mu_0 q(\psi_\epsilon,\phi)-(V_\epsilon,\phi)_{L^2(\Omega)},
\]
for every $\phi \in H^1_{0,K_\epsilon}$, so that, choosing $\phi=A_{K_\epsilon}(\psi_\epsilon)-\mu_0\psi_\epsilon$ in the previous identity, we arrive at
\[
q(A_{K_\epsilon}(\psi_\epsilon)-\mu_0\psi_\epsilon)=
-(V_\epsilon,A_{K_\epsilon}(\psi_\epsilon)-\mu_0\psi_\epsilon)_{L^2(\Omega)}.
\]
The Cauchy-Schwartz inequality, assumption \eqref{eq:c_assumption} and Lemma \ref{lemma:L^2_norm}, together with the previous equality, provide
\[
\big(q(A_{K_\epsilon}(\psi_\epsilon)-\mu_0\psi_\epsilon)\big)^{1/2}
\leq\frac{1}{\sqrt{c_0}} \|V_\epsilon\|_{L^2(\Omega)} =o(C_\epsilon^{1/2})
\]
as $\epsilon\to0$. By combining the last inequality with \eqref{eq:dist_mu0_sigma} and \eqref{eq:q_psi_esp} we see that
\begin{equation}\label{eq:dist_mu0_sigma2}
\dist(\mu_0,\sigma(A_{K_\epsilon}))=o(C_\epsilon^{1/2}) \quad\text{ as }\epsilon\to0.
\end{equation}
We know from Theorem \ref{thm:cont_eigen} that $\lambda_n(\Omega;K_\epsilon)\to\lambda_n$ as $\epsilon\to 0$ and so, since $\lambda_0$ is assumed to be simple, also $\lambda_\epsilon$ is simple for $\epsilon>0$ sufficiently small. Hence, denoting $\mu_\epsilon=\mu_{n_0}(\Omega;K_\epsilon)=1/\lambda_\epsilon$, we have
\[
\dist(\mu_0,\sigma(A_{K_\epsilon}))=|\mu_0-\mu_\epsilon|
\]
for $\epsilon>0$ small enough. Then, using \eqref{eq:dist_mu0_sigma2},
\[
|\lambda_0-\lambda_\epsilon|=\lambda_0\lambda_\epsilon|\mu_0-\mu_\epsilon|
=o(C_\epsilon^{1/2})
\]
as $\epsilon\to0$, so that claim \eqref{eq:claim1} is proved.

\bigskip
	
	Let now $\Pi_\epsilon\colon L^2(\Omega)\to L^2(\Omega)$ be the orthogonal projection onto the one-dimensional eigenspace corresponding to $\lambda_\epsilon$, that is to say
	\[
	\Pi_\epsilon \psi =(\psi,\phi_\epsilon)_{L^2(\Omega)} \phi_\epsilon
	\quad\text{for every }\psi\in L^2(\Omega),
	\]
where we denoted by $\phi_\epsilon$ a $L^2(\Omega)$-normalized eigenfunction associated to $\lambda_\epsilon$.

\smallskip\noindent{\bf Step 2.} We claim that 
\begin{equation}\label{eq:claim2}
q(\psi_\epsilon-\Pi_\epsilon\psi_\epsilon)=o(C_\epsilon)\quad\text{as }\epsilon\to 0.
\end{equation}
	 Let 
\[
\Phi_\epsilon:=\psi_\epsilon-\Pi_\epsilon\psi_\epsilon \quad\text{and}\quad
\xi_\epsilon:=A_{K_\epsilon}(\Phi_\epsilon)-\mu_\epsilon\Phi_\epsilon.
\]	 
Using the fact that $\Pi_\epsilon\psi_\epsilon$ belongs to the eigenspace associated to $\lambda_\varepsilon$ and relation \eqref{eq:psi_eps_phi}, we have, for every $\phi\in H^1_{0,K_\epsilon}$,
	\begin{align*}
q(\Phi_\epsilon,\phi)-\lambda_\epsilon (\Phi_\epsilon,\phi)_{L^2(\Omega)} 
&=q(\psi_\epsilon,\phi)-\lambda_0 (\psi_\epsilon,\phi)_{L^2(\Omega)}  + (\lambda_0-\lambda_\epsilon) (\psi_\epsilon,\phi)_{L^2(\Omega)} \\
&=\lambda_0 (V_\epsilon,\phi)_{L^2(\Omega)} + (\lambda_0-\lambda_\epsilon) (\psi_\epsilon,\phi)_{L^2(\Omega)}.
        \end{align*}
Thanks to the previous relation, with $\phi=\xi_\epsilon$, and the definition of $A_{K_\epsilon}$, we obtain
	\begin{equation}\label{eq:q_xi_eps}
	\begin{aligned}	 
	 q(\xi_\epsilon)=q(A_{K_\epsilon}(\Phi_\epsilon),\xi_\epsilon)-\mu_\epsilon q(\Phi_\epsilon,\xi_\epsilon)
	 &=-\mu_\epsilon \left[q(\Phi_\epsilon,\xi_\epsilon)-\lambda_\epsilon (\Phi_\epsilon,\xi_\epsilon)_{L^2(\Omega)}\right] \\
	 &=-\frac{\lambda_0}{\lambda_\epsilon} (V_\epsilon,\xi_\epsilon)_{L^2(\Omega)}
	 -\frac{\lambda_0-\lambda_\epsilon}{\lambda_\epsilon}(\psi_\epsilon,\xi_\epsilon)_{L^2(\Omega)}.
 	\end{aligned}
 	\end{equation}
Combining \eqref{eq:psi_eps_phi2} and \eqref{eq:q_psi_esp} we obtain that $\norm{\psi_\epsilon}_{L^2(\Omega)}=1+o(1)$ as $\epsilon\to 0$. Therefore, from \eqref{eq:q_xi_eps}, taking into account \eqref{eq:q_psi_esp}, we deduce the existence of a constant $C$ independent from $\epsilon$ such that
\[
\sqrt{q(\xi_\epsilon)} \leq C \left( \|V_\epsilon\|_{L^2(\Omega)}+|\lambda_0-\lambda_\epsilon|\right).
\]
Thus, using the definition of $\xi_\epsilon$, Lemma \ref{lemma:L^2_norm} and \eqref{eq:claim1}, we obtain that
\begin{equation}\label{eq:xi_eps_norm}
q(A_{K_\epsilon}(\Phi_\epsilon)-\mu_\epsilon\Phi_\epsilon)
=o(C_\epsilon) \quad \text{as }\epsilon\to0.
\end{equation}
Let
\[
N_\epsilon=\left\{ w\in H^1_{0,K_\epsilon}: \, (w,\phi_\epsilon)_{L^2(\Omega)}=0 \right\}.
\]
Note that, by definition, $\Phi_\epsilon\in N_\epsilon$. Moreover, being $\phi_\epsilon$ an eigenfunction associated to $\lambda_\epsilon$, from the definition of $A_{K_\epsilon}$ in \eqref{eq:A_K_def} it follows that
\[
A_{K_\epsilon}(w) \in N_\epsilon \quad\text{for every}\quad w\in N_\epsilon.
\]
In particular, the following operator
	\[
		\tilde{A}_\epsilon=A_{K_\epsilon}\restr{N_\epsilon}\colon N_\epsilon\to N_\epsilon
	\]
is well defined. One can easily check that $\tilde{A}_\epsilon$ satisfies properties (i)-(iii) in Lemma \ref{lemma:A_K}; moreover $\sigma(\tilde{A}_\epsilon)=\sigma(A_{K_\epsilon})\setminus \{\mu_\epsilon\}$. In particular, letting $\delta>0$ be such that $(\dist(\mu_\epsilon,\sigma(\tilde{A}_\epsilon)))^2\geq\delta$ for every $\epsilon$ small enough, estimate \eqref{eq:xi_eps_norm}, combined with \eqref{eq:A_K_spectral_theorem}, provides
\[
q(\psi_\epsilon-\Pi_\epsilon\psi_\epsilon) =q(\Phi_\epsilon)
\leq \frac{1}{\delta} (\dist(\mu_\epsilon,\sigma(\tilde{A}_\epsilon)))^2\, q(\Phi_\epsilon)
\leq \frac{1}{\delta} q(\tilde{A}_\epsilon(\Phi_\epsilon)-\mu_\epsilon\Phi_\epsilon)
=o(C_\epsilon)
\]
as $\epsilon\to0$, thus proving claim \eqref{eq:claim2}.

\smallskip\noindent{\bf Step 3.} We claim that 
\begin{equation}\label{eq:claim3}
	\lambda_\epsilon-\lambda_0=C_\epsilon+o(C_\epsilon)\quad\text{as }\epsilon\to 0.
      \end{equation}
	From the definition of $\psi_\epsilon$, Lemma \ref{lemma:L^2_norm}, and the previous step we deduce that
	\[
		\norm{\phi_0-\Pi_\epsilon\psi_\epsilon}_{L^2(\Omega)}\leq \norm{V_\epsilon}_{L^2(\Omega)}+\norm{\psi_\epsilon-\Pi_\epsilon\psi_\epsilon}_{L^2(\Omega)}=o(C_\epsilon^{1/2})\quad\text{as }\epsilon\to 0,
	\]
	which yields both
\begin{equation}\label{eq:Pi_eps_norm}
\norm{\Pi_\epsilon\psi_\epsilon}_{L^2(\Omega)}=1+o(C_\epsilon^{1/2})
\quad\text{as }\epsilon\to 0
\end{equation}	
and, consequently,
\begin{equation}\label{eq:phi_eps_norm}
\norm{\phi_0-\hat\phi_\epsilon}_{L^2(\Omega)}= o(C_\epsilon^{1/2})
\quad\text{as }\epsilon\to 0,
\end{equation}	
 where $\hat\phi_\epsilon=\Pi_\epsilon\psi_\epsilon/\|\Pi_\epsilon\psi_\epsilon\|_{L^2(\Omega)}$.
 Using the fact that $\hat\phi_\epsilon$ is an eigenfunction associated to $\lambda_\epsilon$ and \eqref{eq:psi_eps_phi2} with $\phi=\hat \phi_\epsilon$, we obtain
\begin{equation}\label{eq:psi_eps_phi_eps}
(\lambda_\epsilon-\lambda_0) (\psi_\epsilon,\hat\phi_\epsilon)_{L^2(\Omega)}
=\lambda_0 (V_\epsilon,\hat\phi_\epsilon)_{L^2(\Omega)}.
\end{equation}
But actually 
\begin{equation}\label{eq:V_eps_phi_eps}
\lambda_0(V_\epsilon,\hat\phi_\epsilon)=C_\epsilon+o(C_\epsilon) \quad\text{as }\epsilon\to 0.
\end{equation}
Indeed, since $\psi_\epsilon$ and $V_\epsilon$ are orthogonal with respect to $q$ and $\phi_0$ is an eigenfunction corresponding to $\lambda_0$, we have that
\begin{align*}
C_\epsilon=q(V_\epsilon)&=q(V_\epsilon,\phi_0-\psi_\epsilon)=q(V_\epsilon,\phi_0)=\lambda_0(V_\epsilon,\phi_0)_{L^2(\Omega)} \\
&=\lambda_0 (V_\epsilon,\hat\phi_\epsilon)_{L^2(\Omega)} 
+\lambda_0 (V_\epsilon,\phi_0-\hat\phi_\epsilon)_{L^2(\Omega)},
\end{align*}
so that Lemma \ref{lemma:L^2_norm} and relation \eqref{eq:phi_eps_norm} allow us to prove \eqref{eq:V_eps_phi_eps}. Concerning the left hand side of \eqref{eq:psi_eps_phi_eps} we have, exploiting \eqref{eq:Pi_eps_norm} and \eqref{eq:claim2}, 
\[
(\psi_\epsilon,\hat\phi_\epsilon)_{L^2(\Omega)}
= \frac{(\psi_\epsilon-\Pi_\epsilon\psi_\epsilon,\Pi_\epsilon\psi_\epsilon)_{L^2(\Omega)}+\|\Pi_\epsilon\psi_\epsilon\|_{L^2(\Omega)}^2}{\|\Pi_\epsilon\psi_\epsilon\|_{L^2(\Omega)}}=1+o(C_\epsilon^{1/2}) \quad\text{as }\epsilon\to 0.
\]
By combining the last estimate with \eqref{eq:psi_eps_phi_eps} and \eqref{eq:V_eps_phi_eps}, we complete the proof.
\end{proof}

\section{Set scaling to a boundary point}\label{sec:blow_up}

Hereafter we assume $N\geq 3$. The purpose of this section is to find, in some particular cases, the explicit behaviour of the function $\epsilon\mapsto \Capa_{\bar{\Omega},c}(K_\epsilon,\phi_0)$ and therefore to give a more concrete connotation to the asymptotic expansion proved in Theorem \ref{thm:sharp_asymp}. We consider a particular class of families of concentrating sets, that includes the case in which $K_\epsilon$ is obtained by rescaling a fixed compact set $K$ by a factor $\epsilon>0$. 

First, we prove  Proposition \ref{prop:vanish_phi_0}.

\begin{proof}[Proof of Proposition \ref{prop:vanish_phi_0}]
	The proof is organized as follows: we first  derive \eqref{eq:phi_0_psi_gamma} for a certain diffeomorsphism in the class $\mathcal{C}$ and then we prove that it holds true for any diffeomorphism in $\mathcal{C}$. Hence we start by taking into consideration a particular diffeomorphism $\Phi_{\textup{AE}}\in\mathcal{C}$, first introduced in \cite{Adolfsson1997}.
 Let $g\in C^{1,1}(B_{r_0}')$ be, as in \eqref{eq:hp_boundary_1}, the function that describes $\partial\Omega$ locally near the origin and let $\rho\in C_c^\infty(\R^{N-1})$ be such that $\supp \rho\sub B_1'$, $\rho\geq 0$ in $\R^{N-1}$, $\rho \not\equiv 0$ and $-\nabla\rho (y')\cdot y'\geq 0$ in $\R^{N-1}$. Then, for any $\delta>0$, let
	\begin{equation*}
	\rho_{\delta}(y')=c_\rho^{-1} \delta^{-N+1}\rho\left(\frac{y'}{\delta}\right)\quad\text{with}\quad c_\rho=\int_{\R^{N-1}}\rho(y')\dy',
	\end{equation*}
	be a family of mollifiers. Now, for $j=1,\dots,N-1$ and $y_N>0$, we let
	\begin{equation*}\label{eq:u_j}
	u_j(y',y_N):=y_j-y_N\left( \rho_{y_N}\star \frac{\partial g}{\partial y_j} \right)(y'),
	\end{equation*}
 	where $\star$ denotes the convolution product. Moreover, we define
	\begin{equation*}
	\psi_j(y',y_N):=\begin{cases}
	u_j(y',y_N), &\text{for }y_N>0, \\
	4u_j(y',-\frac{y_N}{2})-3u_j(y',-y_N), &\text{for }y_N<0.
	\end{cases}
	\end{equation*}
	One can prove that $\psi_j\in C^{1,1}(B_{r_0/2})$. Finally, we let $F\colon B_{r_0/2}\to\R^N$ be defined as follows
	\begin{equation*}
	F(y',y_N):=(\psi_1(y',y_N),\dots,\psi_{N-1}(y',y_N),y_N+g(y')).
	\end{equation*}
	Computations show that the Jacobian matrix of $F$ on the hyperplane $\{y_N=0\}$ is as follows
	\begin{equation*}
	J_{F}(y',0)=		\begin{pmatrix}
	1 & 0 & \cdots & 0 & -\frac{\partial g}{\partial y_1}(y') \\
	0 & 1 & \cdots & 0 & - \frac{\partial g}{\partial y_2}(y') \\
	\vdots & \vdots & \ddots & \vdots & \vdots \\
	0 & 0 & \cdots & 1 & -\frac{\partial g}{\partial y_{N-1}}(y') \\
	\frac{\partial g}{\partial y_1}(y') & \frac{\partial g}{\partial y_2}(y') & \cdots & \frac{\partial g}{\partial y_{N-1}}(y') & 1
	\end{pmatrix},
	\end{equation*}
and so $\abs{\det J_{F}(0)}=1+\abs{\nabla g(0)}^2=1$. Hence, by the inverse function theorem, $F$ is invertible in a neighbourhood of the origin: namely there exists $r_1\in (0,r_0/2)$ such that
$F$ is a diffeomorphism of class $C^{1,1}$ from $B_{r_1}$ to $\mathcal U=F(B_{r_1})$ for some $\mathcal U$ open neighbourhood of $0$. Moreover, it is possible to choose $r_1$ sufficiently small
so that 
\begin{align*}
	&F^{-1}(\mathcal U\cap \Omega)=\R^N_+\cap  B_{r_1}= B_{r_1}^+, \\
	&F^{-1}(\mathcal U\cap \partial\Omega)=\partial\R^N_+\cap B_{r_1}=B_{r_1}',
	\end{align*}
	which means that, near the origin, the image of $\Omega$ through $F^{-1}$ has flat boundary (coinciding with $\partial\R^N_+$). 
In particular we have that the diffeomorphism 
\[
\Phi_{\textup{AE}}:\mathcal U\to B_{r_1},\quad   \Phi_{\textup{AE}}:=F^{-1}
\]
 belongs to the class $\mathcal C$ defined in \eqref{eq:diffeo_class}.

For $y\in \Phi_{\textup{AE}}(\mathcal U\cap\Omega)=B_{r_1}^+$, we let $\hat{\phi}_0(y):=\phi_0(\Phi_{\textup{AE}}^{-1}(y))$; from the equation satisfied by $\phi_0$ in $\Omega$, we deduce that
	\begin{equation}\label{eq:hat_phi}
\int_{B_{r_1}^+}(A(y)\nabla\hat{\phi}_0(y)\cdot\nabla\phi(y)+\hat{c}(y)\hat{\phi}_0(y)\phi(y))\dy=\lambda_0\int_{B_{r_1}^+}p(y)\hat{\phi}_0(y)\phi(y)\dy
	\end{equation}
	for all $\phi\in H^1_{0,S_{r_1}^+}(B_{r_1}^+)$, where $S_{r_1}^+:=\partial B_{r_1}\cap \overline{\R^N_+}$ and
	\begin{equation}\label{eq:A_c_p}
	\begin{gathered}
	A(y)=J_{\Phi_{\textup{AE}}}(\Phi_{\textup{AE}}^{-1}(y))J_{\Phi_{\textup{AE}}}(\Phi_{\textup{AE}}^{-1}(y))^T\abs{\det J_{\Phi_{\textup{AE}}}(\Phi_{\textup{AE}}^{-1}(y))}^{-1},\\
	\hat{c}(y)=c(\Phi_{\textup{AE}}^{-1}(y))\abs{\det J_{\Phi_{\textup{AE}}}(\Phi_{\textup{AE}}^{-1}(y))}^{-1},\\
	p(y)=\abs{\det J_{\Phi_{\textup{AE}}}(\Phi_{\textup{AE}}^{-1}(y))}^{-1}.
	\end{gathered}
	\end{equation}
	We point out that equation \eqref{eq:hat_phi} is the weak formulation of the problem
	\begin{equation*}
	\begin{bvp}
	-\dive(A(y)\nabla \hat{\phi}_0(y))+\hat{c}(y)\hat{\phi}_0(y)&=\lambda_0 \,p(y)\hat{\phi}_0(y), &&\text{in }B_{r_1}^+,\\
	\nabla \hat{\phi}_0(y)A(y)\cdot\nnu(y)&=0,&&\text{on }B_{r_1}'.
	\end{bvp}
	\end{equation*}
	One can prove that $A$ is symmetric and uniformly elliptic in $B_{r_1}^+$ (if $r_1$ is choosen sufficiently small); moreover, if we denote $A(y)=(a_{i,j}(y))_{i,j=1,\dots,N}$, then $a_{i,j}\in C^{0,1}(B_{r_1}^+\cup B_{r_1}')$ and
	\begin{equation}\label{eq:A_coeff}
	\begin{aligned}
	a_{i,i}(y',0)&=1+\abs{\nabla g(y')}^2-\left( \frac{\partial g}{\partial y_i}(y') \right)^2, &&\text{for all }i=1,\dots,N-1, \\
	a_{i,j}(y',0)&=-\frac{\partial g}{\partial y_i}(y') \frac{\partial g}{\partial y_j}(y'), &&\text{for all }i,j=1,\dots,N-1,~i\neq j, \\
	a_{i,N}(y',0)&=0, &&\text{for all }i=1,\dots,N-1, \\	 
	a_{N,N}=1.
	\end{aligned}
	\end{equation}
	Therefore, if we consider an even reflection of $\hat{\phi}_0$ 
(which we still denote as $\hat{\phi}_0$)
through the hyperplane $\{x_N=0\}$ in $B_{r_1}$, then it satisfies, in this ball, an elliptic equation in divergence form with Lipschitz continuous second order coefficients. More in particular $\hat{\phi}_0$ weakly satisfies
	\begin{equation}\label{eq:reflected}
	-\dive (\bar{A}(y)\nabla\hat{\phi}_0(y))=h(y)\hat{\phi}_0(y)\quad\text{in }B_{r_1}
	\end{equation}
	where
	\[
	\bar{A}(y):=\begin{cases}
	A(y_1,\dots,y_{N-1},y_N), &\text{if }y_N>0, \\
	QA(y_1,\dots,y_{N-1},-y_N) Q, &\text{if }y_N<0,
	\end{cases}
	\]
	with
	\[
	Q:=\begin{pmatrix}
	1 & 0 & \cdots & 0 & 0 \\
	0 & 1 & \cdots & 0 & 0 \\
	\vdots & \vdots & \ddots & \vdots &\vdots \\
	0 & 0 & \cdots & 1 & 0 \\
	0 &0 &\cdots & 0 & -1 
	\end{pmatrix}
	\]
	and $h\in L^\infty(B_{r_1})$. We point out that Lipschitz continuity of the coefficients of the matrix $\bar{A}$ comes from the fact that $a_{i,j}\in C^{0,1}(B_{r_1}^+\cup B_{r_1}')$ and that $a_{i,N}(y',0)=0$ for all $i<N$. Hence we deduce, from \cite{Robbiano1988}, that there exists a homogeneous harmonic polynomial $\psi_\gamma$ of degree $\gamma\in \N$ such that 
\[
\hat{\phi}_0(y)=\psi_\gamma(y)+R_\gamma(y)
\]
where $\|R_\gamma\|_{H^1(B_r)}=O(r^{\gamma+\frac N2+1-\delta})$ for some $\delta\in (0,1)$ as $r\to 0$. In particular
\eqref{eq:phi_0_psi_gamma} holds.
	
	Now let $\Phi\in \mathcal{C}$. We can rewrite
	\begin{equation*}
		\phi_0(\Phi^{-1}(\epsilon x))=\phi_0\left(\Phi_{\textup{AE}}^{-1} \left(\epsilon G_\epsilon(x)\right)\right),\quad\text{with }G_\epsilon(x):=\frac{(\Phi_{\textup{AE}}\circ \Phi^{-1})(\epsilon x)}{\epsilon}.
	\end{equation*}
	Thanks to regularity properties of $\Phi_{\textup{AE}}$ and $\Phi$, we have that
	\[
		G_\epsilon(x)=x+\abs{x}^2 O(\epsilon),\quad\text{as }\epsilon\to 0.
	\]
	As a consequence, one can prove that
	\[
		\epsilon^{-\gamma}\phi_0\left(\Phi_{\textup{AE}}^{-1} \left(\epsilon G_\epsilon(x)\right)\right)\to \psi_\gamma(x),\quad\text{in }H^1(B_R^+)\text{ as }\epsilon\to 0,
	\]
	for all $R>0$, thus concluding the proof of \eqref{eq:phi_0_psi_gamma}.

The proofs of \eqref{eq:1} and \eqref{eq:2} follow from 
\eqref{eq:phi_0_psi_gamma} by making a change of variable in the integral and taking into account that, for all $R>0$, 
\[
\chi_{\epsilon^{-1}\Phi(\Omega\cap B_{R\epsilon})}\to
\chi_{B_R^+}\quad\text{a.e. in }\R^N,
\]
with $\chi_A$ denoting as usual the characteristic function of a set $A\subset \R^N$,
as one can easily deduce from the fact that $\Phi^{-1}(y)=y+O(|y|^2)$ as $y\to 0$.
\end{proof}

In this section we consider a particular class of families of compact sets concentrating to the origin (which is assumed to belong to $\partial\Omega$), as described in Section \ref{sec:sets-scal-bound}. Let us fix $\Phi\in\mathcal{C}$ with 
\begin{equation}\label{eq:3}
\Phi:\mathcal U_0\to B_{R_0}
\end{equation}
being $\mathcal U_0$ an open neighbourhood of $0$ and $R_0>0$ such that $\Phi(\mathcal U_0\cap \Omega)=
B_{R_0}^+$ and $\Phi(\mathcal U_0\cap \partial\Omega)=B_{R_0}'$.
In the rest of this section, we will use the same notation as in the proof of
 Proposition \ref{prop:vanish_phi_0} defining 
 $A$ and $\hat{c}$  as in \eqref{eq:A_c_p} (with $\Phi$ instead of $\Phi_{\textup{AE}}$).

Since $A\in C^{0,1}(B_{R_0}^+,{\mathcal M}_{N\times N})$ (with ${\mathcal
  M}_{N\times N}$ denoting the space of $N\times N$  real matrices) 
 and $A(0)=I_N$,
it is possibile to choose $R_0>0$ small enough in order to have 
\[
\|A(x)-I_N\|_{{\mathcal M}_{N\times N}}\leq \frac12
\quad\text{and}\quad 
\hat{c}(x)\geq \frac{c_0}2\quad \text{for a.e. $x\in B_{R_0}^+$}
\]
 (with $\|\cdot\|_{{\mathcal M}_{N\times N}}$ denoting the operator norm  on ${\mathcal M}_{N\times N}$). With this choice of $R_0$, we have that 
\begin{align}\label{eq:4}
  \int_{B_{R_0/\epsilon}^+}\!\!A(\epsilon x)\nabla u(x)\cdot\nabla u(x)\dx&=
  \int_{B_{R_0/\epsilon}^+}\!\!(A(\epsilon x)-I_N)\nabla u(x)\cdot\nabla u(x)\dx+  \int_{B_{R_0/\epsilon}^+}\!\!|\nabla u(x)|^2\dx\\
&\notag \geq \frac12\, \int_{B_{R_0/\epsilon}^+}\!\!|\nabla u(x)|^2\dx
\end{align}
and 
\begin{equation}\label{eq:6}
  \int_{B_{R_0/\epsilon}^+}\!\!\hat{c}(\epsilon x)u^2(x)\dx\geq \frac{c_0}2   \int_{B_{R_0/\epsilon}^+}\!\!u^2(x)\dx
\end{equation}
for all $u\in H^1(B_{R_0/\epsilon}^+)$.

Let $K_\epsilon\sub \overline{\Omega}\cap \mathcal U_0$ be a compact set for any $\epsilon\in (0,1)$ such that \eqref{eq:hp_blow_up_1_intr} and \eqref{eq:hp_blow_up_2_intr} hold. 
In the following we denote 
\[
\tilde{K}_\epsilon:=\Phi(K_\epsilon)/\epsilon.
\]
For any compact set $H\sub\R^N$ we define the \emph{radius} of $H$ as follows
\begin{equation}\label{eq:radius_H_def}
r(H):=\max_{x\in H}\abs{x}.
\end{equation}
\begin{remark}\label{rmk:equiv_mosco}
	Concerning hypothesis \eqref{eq:hp_blow_up_2_intr}, one can prove that the convergence of $\R^N\setminus \tilde{K}_\epsilon$ to $\R^N\setminus K$ in the sense of Mosco, as $\epsilon\to 0$, introduced in Definition \ref{def:mosco_D}, is equivalent to the convergence of the space $\HO{B_R^+}{\tilde{K}_\epsilon}$ to the space $\HO{B_R^+}{K}$ in the sense of Mosco for all $R>r(M)$. We recall that $\HO{B_R^+}{\tilde{K}_\epsilon}$ is said to \emph{converge to $\HO{B_R^+}{K}$ in the sense of Mosco} if the following holds:
	\begin{enumerate}
		\item[(1)] the weak limit points (as $\epsilon\to 0$) in $H^1(B_R^+)$ of every family of functions $\{u_\epsilon\}_{\epsilon}\subseteq H^1(B_R^+)$, such that $u_\epsilon\in\HO{B_R^+}{\tilde{K}_\epsilon} $ for every $\epsilon$, belong to $\HO{B_R^+}{K}$;
		\item[(2)] for every $u\in\HO{B_R^+}{K}$ there exists a family $\{u_\epsilon\}_\epsilon\subseteq H^1(B_R^+)$ such that $u_\epsilon\in \HO{B_R^+}{\tilde{K}_\epsilon}$ for every $\epsilon$ and $u_\epsilon\to u$ in $H^1(B_R^+)$, as $\epsilon\to 0$.
	\end{enumerate}
The proof of this equivalence is essentially based on the continuity of the extension operator for functions in $H^1(B_R^+)$ and of the restriction operator on $B_R^+$ for functions in $H^1(\R^N)$. Analogously, one can also prove that they are both also equivalent to the convergence of  $\HO{B_R^+}{\tilde{K}_\epsilon\cup S_R^+}$ to  $\HO{B_R^+}{K\cup S_R^+}$ in the sense of Mosco, as $\epsilon\to 0$. These equivalent hypotheses turn out to be more adequate for our needs in this final part of the section.
\end{remark}

In view of Proposition \ref{prop:vanish_phi_0} it is natural to consider the following rescalings of the limit eigenfunction $\phi_0$ and of the $\phi_0$-capacitary potential of $K_\epsilon$
\begin{equation}\label{eq:def_tilde_phi_V}
\tilde{\phi}_\epsilon(y):=\frac{\phi_0(\Phi^{-1}(\epsilon y))}{\epsilon^\gamma
}=\frac{\hat{\phi}_0(\epsilon y)}{\epsilon^{\gamma}},\quad\tilde{V}_\epsilon(y):=\frac{V_{K_\epsilon,\phi_0}(\Phi^{-1}(\epsilon y))}{\epsilon^\gamma},\quad y\in B_{R_0/\epsilon}^+,
\end{equation}
where $\hat{\phi}_0(z)=\phi_0(\Phi^{-1}(z)) \in H^1(B_{R_0}^+)$. 
We have that  $\tilde{\phi}_\epsilon\in H^1(B_{R_0/\epsilon}^+)$, $\tilde{V}_\epsilon-\tilde{\phi}_\epsilon\in\HO{B_{R_0/\epsilon}^+}{\tilde{K}_\epsilon}$, and they satisfy
\begin{equation}\label{eq:blow_up_eq_phi}
\int_{B_{R_0/\epsilon}^+}(A(\epsilon y)\nabla\tilde{\phi}_\epsilon(y)\cdot\nabla\phi(y)+\epsilon^2\hat{c}(\epsilon y)\tilde{\phi}_\epsilon(y)\phi(y))\dy
=\lambda_0\epsilon^2\int_{B_{R_0/\epsilon}^+}p(\epsilon y)\tilde{\phi}_\epsilon(y)\phi(y)\dy,
\end{equation}
for all $\phi\in H^1_{0,S_{R_0/\epsilon}^+}(B_{R_0/\epsilon}^+)$
and 
\begin{equation}\label{eq:blow_up_eq_V}
\int_{B_{R_0/\epsilon}^+}(A(\epsilon y)\nabla\tilde{V}_\epsilon(y)\cdot\nabla\phi(y)+\epsilon^2\hat{c}(\epsilon y)\tilde{V}_\epsilon(y)\phi(y))\dy=0,
\end{equation}
for all $\phi\in \HO{B_{R_0/\epsilon}^+}{\tilde{K}_\epsilon \cup S^+_{R_0/\epsilon}}$.

The following Lemma provides a first, rough estimate of the boundary Sobolev capacity appearing in the asymptotic expansion in Theorem \ref{thm:sharp_asymp}.

\begin{lemma}\label{lemma:big_O_cap}
	We have that
	\[
	\Capa_{\bar{\Omega},c}(K_\epsilon,\phi_0)=O(\epsilon^{N+2\gamma-2}),\quad\text{as }\epsilon\to 0.
	\]
\end{lemma}
\begin{proof}
Recall the definition of $M$ in \eqref{eq:hp_blow_up_1_intr} and that of $r(M)$ in \eqref{eq:radius_H_def}.
Since $K_\epsilon\subseteq \Phi^{-1}(\epsilon \overline{B_{r(M)}})$ and $\Phi^{-1}(y)=y+O(|y|^2)$ as $|y|\to0$,
there exists $R>0$ such that $K_\epsilon\subset B_{R\epsilon}$ for $\epsilon$ sufficiently small.
Now let $\eta_\epsilon\in C_c^\infty(\R^N)$ be such that
	\begin{equation*}
	\begin{gathered}
	0\leq \eta_\epsilon(y)\leq 1, \\
	\abs{\nabla \eta_\epsilon}\leq \frac{2}{\epsilon R},
	\end{gathered}
	\qquad\quad
	\eta_\epsilon(y)=\begin{cases}
	0, &\text{for }y\in\R^N\setminus B_{2\epsilon R}, \\
	1, &\text{for }y\in B_{\epsilon R}.
	\end{cases}
	\end{equation*}
Since $\eta_\epsilon\varphi_0\in \HO{\Omega}{K_\epsilon}$, from \eqref{eq:def_f_capacity} we have that 
\begin{align*}
\Capa_{\bar{\Omega},c}(K_\epsilon,\varphi_0)&\leq q(\eta_\epsilon\varphi_0)\\
&\leq 2\int_{\Omega\cap B_{2R\epsilon}}|\nabla \eta_\epsilon(x)|^2\varphi_0(x)\dx\\
&\quad\quad
+2\int_{\Omega\cap B_{2R\epsilon}}|\eta_\epsilon(x)|^2|\nabla\varphi_0(x)|^2 \dx+\|c\|_{L^\infty(\Omega)}
\int_{\Omega\cap B_{2R\epsilon}}\varphi_0^2(x)\dx\\
&\leq \left(\frac8{\epsilon^2R^2}+\|c\|_{L^\infty(\Omega)}\right) \int_{\Omega\cap B_{2R\epsilon}}\varphi_0^2(x)\dx+2
\int_{\Omega\cap B_{2R\epsilon}}|\nabla\varphi_0(x)|^2 \dx
\end{align*}
The conclusion then follows from \eqref{eq:1} and \eqref{eq:2}. 
\end{proof}

The following lemma, whose proof is classical, is useful in order to pass from a global scale (functions in $\mathcal{D}^{1,2}(\overline{\R^N_+})$) to a local one (meaning functions in $H^1(B_R^+)$) and vice versa.

\begin{lemma}\label{lemma:H_implies_D}
	Let $K\sub\overline{\R^N_+}$ be compact. If $f\in \mathcal{D}^{1,2}(\overline{\R^N_+})$ is such that $f\restr{B_R^+}\in \HO{B_R^+}{K}$ for some $R>r(K)$, then $f\in \mathcal{D}^{1,2}(\overline{\R^N_+}\setminus K)$. Conversely, if $f\in \mathcal{D}^{1,2}(\overline{\R^N_+}\setminus K)$, then $f\restr{B_R^+}\in \HO{B_R^+}{K}$ for all $R>r(K)$.
\end{lemma}

In the lemma below we compare the two notions of capacity arising in our work, namely Definition \ref{def:capacity} and Definition \ref{def:rel_capacity}.

\begin{lemma}[Comparison of capacities]\label{lemma:comparison_capa}
	Let $K\sub \overline{\R^N_+}$ be a compact set and let $R>r(K)$. Then there exists a constant $\alpha=\alpha(R)>0$ such that
	\[
	\alpha^{-1}\,\mathrm{cap}_{\overline{\R^N_+}}(K)\leq \Capa_{\overline{B_R^+}}(K)\leq \alpha\,\mathrm{cap}_{\overline{\R^N_+}}(K).
	\]
\end{lemma}
\begin{proof}
	Let $\eta_K\in C_c^\infty(\R^N)$ be such that $\eta_K=1$ in a neighbourhood of $K$.

 In order to prove the left hand inequality, let $W_R\in H^1(B_R^+)$ be the potential achieving $\Capa_{\overline{B_R^+}}(K)$ and let $\hat{W}_R\in H^1(\R^N_+)$ be its extension to $\R^N_+$. Obviously $\hat{W}_R\in \mathcal{D}^{1,2}(\overline{\R^N_+})$ and, by Lemma \ref{lemma:H_implies_D}, we have that $\hat{W}_R-\eta_K\in\mathcal{D}^{1,2}(\overline{\R^N_+}\setminus K)$. Therefore $\hat{W}_R$ is admissible for $\mathrm{cap}_{\overline{\R^N_+}}(K)$ and hence
\begin{align*}
	\mathrm{cap}_{\overline{\R^N_+}}(K)&\leq \int_{\R^N_+}|\nabla \hat{W}_R|^2\dx\leq \int_{\R^N_+}( |\nabla \hat{W}_R  |^2+\hat{W}_R^2 )\dx \\
	&\leq C_1(R)\int_{B_R^+}( \abs{\nabla W_R}^2+W_R^2 )\dx 
=C_1(R)\Capa_{\overline{B_R^+}}(K),
	\end{align*}
where $C_1(R)$ is the constant related to the extension operator for Sobolev functions.
 In order to prove the other inequality, let $W_K\in \mathcal{D}^{1,2}(\overline{\R^N_+})$ be the potential achieving $\mathrm{cap}_{\overline{\R^N_+}}(K)$. Since $W_K-\eta_K\in \mathcal{D}^{1,2}(\overline{\R^N_+}\setminus K)$, then $W_K\in H^1(B_R^+)$ and, in view of Lemma \ref{lemma:H_implies_D}, $W_K-1\in \HO{B_R^+}{K}$. Hence $W_K$ is admissible for $\Capa_{\overline{B_R^+}}(K)$. 
 Moreover, by H\"older and Sobolev inequalities, we have that
	\[
	\norm{W_K}_{L^2(B_R^+)}^2\leq  
\|W_K\|_{L^{2^*}(B_R^+)}^2|B_R^+|^{2/N}\leq 
C_2(R)\norm{\nabla W_K}_{L^{2}(\R^N_+)}^2,
	\]
	for some $C_2(R)>0$ depending only on $N$ and $R$. Then the following estimates hold
\begin{align*}
	\Capa_{\overline{B_R^+}}(K)\leq \int_{B_R^+}(\abs{\nabla W_K}^2+W_K^2  )\dx \leq (1+C_2(R))\,\mathrm{cap}_{\overline{\R^N_+}}(K),
	\end{align*} 
thus concluding the proof.
\end{proof}

In order to prove Theorem \ref{thm:blow_up} the following Poincar\'e type inequality is needed.

\begin{lemma}[Poincar\'e Inequality]\label{lemma:poincare}
	Let $M,K\sub \overline{\R^N_+}$ and 
$\{K_\epsilon\}_{\epsilon\in(0,1)}$ satisfy \eqref{eq:hp_blow_up_1_intr}--\eqref{eq:hp_blow_up_2_intr} 
for some $\Phi\in \mathcal C$ and let 
$\tilde{K}_\epsilon:=\Phi(K_\epsilon)/\epsilon$. Let us assume that $\mathrm{cap}_{\overline{\R^N_+}}(K)>0$. For any $R>r(M)$ there exist $\epsilon_0\in(0,1)$ and $C>0$ (both depending on
 $R$ and $K$) such that
	\[
	\int_{B_R^+}u^2\dx\leq C\int_{B_R^+}\abs{\nabla u}^2\dx
	\]
	for all $u\in \HO{B_R^+}{\tilde{K}_\epsilon}$ and for all $\epsilon<\epsilon_0$.
\end{lemma}
\begin{proof}
 By way of contradiction, suppose that, for a certain $R>r(M)$, there exist a sequence of real numbers $\epsilon_n\to 0^+$ and a sequence of functions $u_n\in \HO{B_R^+}{\tilde{K}_{\epsilon_n}}$ such that
	\[
	\int_{B_R^+}u_n^2\dx>n\int_{B_R^+}\abs{\nabla u_n}^2\dx.
	\]
	Now let us consider the sequence
	\[
	v_n:=\frac{u_n}{\norm{u_n}_{L^2(B_R^+)}},
	\]
	so that
	\[
	\norm{v_n}_{L^2(B_R^+)}=1\quad\text{and}\quad
 \int_{B_R^+}\abs{\nabla v_n}^2\dx<\frac{1}{n}.
	\]
	Therefore, since $\norm{v_n}_{H^1(B_R^+)}$ is uniformly bounded with respect to $n$, there exists $v\in H^1(B_R^+)$ such that, up to a subsequence, $v_n\rightharpoonup v$ weakly in $H^1(B_R^+)$. By compactness $v_n\to v$ strongly in $L^2(B_R^+)$; this implies that $\norm{v}_{L^2(B_R^+)}=1$ and hence that $v\not\equiv 0$. On the other hand, by weak lower semicontinuity, 
	\[
	\int_{B_R^+}\abs{\nabla v}^2\dx \leq \liminf_{n\to\infty}\int_{B_R^+}\abs{\nabla v_n}^2\dx=0,
	\]
	hence there exists a constant $\kappa\neq 0$ such that $ v=\kappa$ a.e. in $B_R^+$.
	Finally, since $\R^N\setminus\tilde{K}_{\epsilon_n}$ is converging to $\R^N\setminus K$ in the sense of Mosco, $v\in \HO{B_R^+}{K}$ (see Remark \ref{rmk:equiv_mosco}) and  this implies the existence of a sequence $\{w_n\}_n\subset C^\infty_c(\overline{B_R^+}\setminus K)$ such that $\norm{w_n-\kappa}_{H^1(B_R^+)} \to0$ as $n\to+\infty$. Letting $z_n=(\kappa-w_n)/\kappa$, we have that
	\[
	z_n -1 \in \HO{B_R^+}{K} \quad\text{and}\quad \norm{z_n}_{H^1(B_R^+)}\to0
	\]
	as $n\to+\infty$, thus implying 
 $\Capa_{\overline{B_R^+}}(K)=0$ and hence
contradicting, in view of Lemma \ref{lemma:comparison_capa}, the fact that $\mathrm{cap}_{\overline{\R^N_+}}(K)>0$.
\end{proof}

In the same spirit of Proposition \ref{prop:def_V_K_f}, we have that the relative $\psi_\gamma$-capacity of the set $K$, denoted by $\mathrm{cap}_{\overline{\R^N_+}}(K,\psi_\gamma)$ (see Definition \ref{def:rel_capacity}), is uniquely achieved, as asserted in the following lemma.

\begin{lemma}\label{lemma:limit_pro} 
Let $\eta\in C_c^\infty(\overline{\R^N_+})$ be a cut-off function such that $\eta=1$ in a neighbourhood of $K$. There exists a unique $\tilde{V}\in\DC$ such that 
$\tilde{V}-\eta\psi_\gamma\in \DCK$ and
	\[
	\int_{\R^N_+}\nabla\tilde{V}\cdot\nabla\phi\dx=0\quad\text{for all }\phi\in\DCK,
	\]
i.e. weakly solving
	\[
	\begin{bvp}
	-\Delta \tilde{V} & =0, & & \text{in }\R^N_+\setminus K,\\
	\tilde{V} & =\psi_\gamma, & & \text{on }K,\\
	\frac{\partial\tilde{V}}{\partial\nnu} & =0, & & \text{on }\partial\R^N_+\setminus K.
	\end{bvp}
	\]
Moreover
	\[
	\mathrm{cap}_{\overline{\R^N_+}}(K,\psi_\gamma)=\int_{\R^N_+}|\nabla \tilde{V}|^2\dx
	\]
and $\tilde{V}$ does not depend on the choice of the cut-off function $\eta$.
\end{lemma}

Since we are in the case $N\geq 3$, the following Hardy-type inequality on half balls holds.

\begin{lemma}[Hardy-type inequality]\label{lemma:hardy}
For all $R>0$ and $u\in H^1(B_R^+)$
	\[
	\frac{N-2}{2}\int_{B_R^+}\frac{u^2}{\abs{x}^2}\dx\leq \frac{N+1}{R^2}\int_{B_R^+}u^2\dx+\frac{N}{N-2}\int_{B_R^+}\abs{\nabla u}^2\dx.
	\]
	\end{lemma}
\begin{proof}
	By integrating over $B_R^+$ the identity
	\[
	\dive \left(u^2\frac{x}{\abs{x}^2}\right)=(N-2)\frac{u^2}{\abs{x}^2}+2u\nabla u\cdot\frac{x}{\abs{x}^2},
	\]
	we obtain that, for any $u\in C^\infty(\overline{B_R^+})$,
	\begin{align*}
	(N-2)\int_{B_R^+}\frac{u^2}{\abs{x}^2}\dx&=\int_{\partial B_R^+}u^2\frac{x\cdot \nnu}{\abs{x}^2}\ds-2\int_{B_R^+}u\nabla u\cdot \frac{x}{\abs{x}^2}\dx \\
	&\leq \frac{1}{R}\int_{S_R^+}u^2\ds+\frac{N-2}{2}\int_{B_R^+}\frac{u^2}{\abs{x}^2}\dx +\frac{2}{N-2}\int_{B_R^+}\abs{\nabla u}^2\dx,
	\end{align*}
	thanks to the fact that $x\cdot\nnu =0$ on $\{x_1=0\}\cap\partial B_R^+$ and $x=R\,\nnu$ on $S_R^+$.

 On the other hand, integrating over $B_R^+$ the identity
	\[
	\dive (u^2 x)=2u\nabla u\cdot x+Nu^2
	\]
	and arguing in a similar way, we deduce that
	\[
	\int_{S_R^+}u^2\ds\leq \frac{N+1}{R}\int_{B_R^+}u^2\dx+R\int_{B_R^+}\abs{\nabla u}^2\dx.
	\]
	Combining those two relations the lemma is proved.
\end{proof}

The following proposition provides a blow-up analysis for scaled capacitary potentials, which will be the core of the proof of Theorem \ref{thm:blow_up}.

\begin{proposition}\label{prop:blow_up}
Let $\left\{K_\epsilon\right\}_{\epsilon>0}\sub\overline{\Omega}$ be a family of compact sets concentrating at $\{0\}\sub\partial\Omega$ as $\epsilon\to 0$ and satisfying \eqref{eq:hp_blow_up_1_intr}-\eqref{eq:hp_blow_up_2_intr} for some $\Phi\in\mathcal{C}$ and  for some compact sets
$M,K\sub\overline{\R^N_+}$ with $\mathrm{cap}_{\overline{\R^N_+}}(K)>0$. 
Let $\phi_0$ be as in \eqref{eq:phi0} and 
let $\gamma$, $\psi_\gamma$ be as in \eqref{eq:phi_0_psi_gamma}-\eqref{eq:psi_gamma}.
Let $\tilde{V}_\epsilon$ be as in \eqref{eq:def_tilde_phi_V} and $\tilde{V}$ as in Lemma \ref{lemma:limit_pro}. Then
\begin{align*}
	\tilde{V}_\epsilon\rightharpoonup \tilde{V}\quad&\text{weakly in }H^1(B_R^+),\\
	A(\epsilon x)\nabla\tilde{V}_\epsilon(x) \rightharpoonup \nabla\tilde{V}(x)\quad&\text{weakly in }L^2(B_R^+),\\
	\epsilon^2\hat{c}(\epsilon x)\tilde{V}_\epsilon(x)\rightharpoonup 0\quad&\text{weakly in }L^2(B_R^+),
\end{align*}
	as $\epsilon\to 0$ for all $R>r(M)$, where
 $A$ and $\hat{c}$ are as in \eqref{eq:A_c_p} (with $\Phi$ instead of $\Phi_{\textup{AE}}$).
\end{proposition}
\begin{proof}
	From Lemma \ref{lemma:L^2_norm} and Lemma \ref{lemma:big_O_cap} we have that
	\[
	\int_\Omega\abs{\nabla V_{K_\epsilon,\phi_0}}^2\dx=\Capa_{\bar{\Omega},c}(K_\epsilon,\phi_0)(1+o(1))=O(\epsilon^{N+2\gamma-2}),
	\]
	as $\epsilon\to 0$. On the other hand, 
letting $\mathcal U_0$ and $R_0$ be as in \eqref{eq:3},
by a change of variables we have that
	\[
	\int_\Omega\abs{\nabla V_{K_\epsilon,\phi_0}}^2\dx \geq 
	\int_{\Omega\cap \mathcal U_0}\abs{\nabla V_{K_\epsilon,\phi_0}}^2\dx
=\epsilon^{N+2\gamma-2}\int_{B_{R_0/\epsilon}^+}A(\epsilon x)\nabla\tilde{V}_\epsilon(x)\cdot\nabla\tilde{V}_\epsilon(x)\dx
	\]
	and so
	\begin{equation}\label{eq:V_tilde_9}
	\int_{B_{R_0/\epsilon}^+}A(\epsilon x)\nabla\tilde{V}_\epsilon(x)\cdot\nabla\tilde{V}_\epsilon(x)\dx=O(1)
	\end{equation}
	as $\epsilon\to 0$.
From \eqref{eq:V_tilde_9} and \eqref{eq:4} it
follows that
\begin{equation}\label{eq:5}
 \int_{B_{R_0/\epsilon}^+}|\nabla \tilde{V}_\epsilon(x)|^2\dx\leq
2 \int_{B_{R_0/\epsilon}^+}A(\epsilon x)\nabla \tilde{V}_\epsilon(x)\cdot \nabla \tilde{V}_\epsilon(x)\dx=O(1)\quad\text{as $\epsilon\to 0$}.
\end{equation}
On the other hand, in view of \eqref{eq:6}, 
\begin{align*}
  \Capa_{\bar{\Omega},c}(K_\epsilon,\phi_0)&\geq 
\int_{\Omega\cap \mathcal U_0}c(x)|V_{K_\epsilon,\phi_0}|^2\dx\\
&=\epsilon^{N+2\gamma}\int_{B_{R_0/\epsilon}^+}\hat{c}(\epsilon x)|\tilde{V}_\epsilon(x)|^2\dx\geq \frac{c_0}2\,\epsilon^{N+2\gamma}\int_{B_{R_0/\epsilon}^+}|\tilde{V}_\epsilon(x)|^2\dx
\end{align*}
so that from Lemma 
\ref{lemma:big_O_cap} we deduce that 
\begin{equation}\label{eq:7}
\epsilon^2\int_{B_{R_0/\epsilon}^+}|\tilde{V}_\epsilon(x)|^2\dx=O(1)
  \quad\text{as }\epsilon\to 0.
\end{equation}
From \eqref{eq:5} we deduce that there exists $C_1>0$ (not depending on $R$) such that, for
every $R>0$,
	\begin{equation}\label{eq:V_tilde_1}
	\int_{B_R^+}|\nabla\tilde{V}_\epsilon|^2\dx\leq C_1
\quad \text{for all }\epsilon\in(0,R_0/R).
	\end{equation}
	From the Poincar\'e inequality proved in Lemma \ref{lemma:poincare}, we deduce that, for every $R>r(M)$, there exist $C_2=C_2(R)>0$ and $\epsilon_{1,R}>0$ (both depending on $R$) such that, if $\epsilon\in(0,\epsilon_{1,R})$,
	\begin{align*}
	\int_{B_R^+}|\tilde{V}_\epsilon|^2\dx&\leq 2\int_{B_R^+}|\tilde{V}_\epsilon-\tilde{\phi}_\epsilon|^2\dx+2\int_{B_R^+}|\tilde{\phi}_\epsilon|^2\dx \\
	&\leq 2C_2\int_{B_R^+}|\nabla(\tilde{V}_\epsilon-\tilde{\phi}_\epsilon)|^2\dx+2\int_{B_R^+}|\tilde{\phi}_\epsilon|^2\dx\\
	&\leq 4C_2\int_{B_R^+}(|\nabla\tilde{V}_\epsilon|^2+|\nabla\tilde{\phi}_\epsilon|^2)\dx+2\int_{B_R^+}|\tilde{\phi}_\epsilon|^2\dx.
	\end{align*}
	Hence, from \eqref{eq:phi_0_psi_gamma} and \eqref{eq:V_tilde_1} we have that, for every $R>r(M)$,
 there exist $C_3=C_3(R)>0$ and $\epsilon_{2,R}>0$ (both depending on $R$)  such that, if $\epsilon\in(0,\epsilon_{2,R})$,
	\begin{equation}\label{eq:V_tilde_2}
	\int_{B_R^+}|\tilde{V}_\epsilon|^2\dx\leq C_3.
	\end{equation}
	Combining \eqref{eq:V_tilde_1} and \eqref{eq:V_tilde_2} with a diagonal process, we deduce that there exists $W\in H^1_{\textup{loc}}(\R^N_+)$ (not depending on $R$) such that, along a subsequence $\epsilon=\epsilon_n\to0^+$,
\begin{equation}\label{eq:V_tilde_3}
\begin{split}
  \tilde{V}_\epsilon\rightharpoonup W\quad&\text{weakly in }H^1(B_R^+), \\
  \tilde{V}_\epsilon\to W\quad&\text{strongly in }L^2(B_R^+),\\
  \tilde{V}_\epsilon\to W\quad&\text{a.e. in }\R^N_+,
\end{split}
\end{equation}
	for all $R>r(M)$. 
Since $c$ is bounded and  
$\|A(\epsilon x)-I_N\|_{{\mathcal M}_{N\times N}}\leq C\,R\,\epsilon$ for all $x\in B_{R}^+$, from \eqref{eq:V_tilde_3} it follows easily that  
	\begin{gather}
	A(\epsilon x)\nabla \tilde{V}_\epsilon(x)\rightharpoonup \nabla W(x)\quad\text{weakly in }L^2(B_R^+), \label{eq:V_tilde_4} \\
	\epsilon^2\hat{c}(\epsilon x)\tilde{V}_\epsilon(x)\rightharpoonup 0\quad\text{weakly in }L^2(B_R^+),\label{eq:V_tilde_5}
	\end{gather}
	as $\epsilon=\epsilon_n\to 0$, for all $R>r(M)$.

	Now let $\phi\in C_c^\infty(\overline{\R^N_+}\setminus K)$. Then there exists $R>0$ such that 
	\[
	\phi\in C_c^\infty(B_R^+\cup (B_R'\setminus K))\sub \HO{B_R^+}{K\cup S_R^+}.
	\]
	Since $\HO{B_R^+}{\tilde{K}_\epsilon\cup S_R^+}$ is converging to $\HO{B_R^+}{K\cup S_R^+}$ in the sense of Mosco (see hypothesis \eqref{eq:hp_blow_up_2_intr} and Remark \ref{rmk:equiv_mosco}), there exists a sequence $\psi_\epsilon\in \HO{B_R^+}{\tilde{K}_\epsilon\cup S_R^+}$ such that
	\begin{equation}\label{eq:V_tilde_6}
	\psi_\epsilon\to \phi\quad\text{strongly in }H^1(B_R^+),~\text{as }\epsilon\to 0.
	\end{equation}
	From the equation \eqref{eq:blow_up_eq_V} satisfied by $\tilde{V}_\epsilon$ we know that, for all $\epsilon\in (0,R_0/R)$ (so that $B_R^+\subset B_{R_0/\epsilon}^+$)
	\[
	\int_{B_R^+}(A(\epsilon x)\nabla\tilde{V}_\epsilon(x)\cdot\nabla\psi_\epsilon(x)+\epsilon^2\hat{c}(\epsilon x)\tilde{V}_\epsilon(x)\psi_\epsilon(x))\dx=0.
	\]
	Therefore, combining \eqref{eq:V_tilde_4},\eqref{eq:V_tilde_5} and \eqref{eq:V_tilde_6} we obtain that
	\[
	\int_{B_R^+}\nabla W\cdot\nabla \phi\dx=0.
	\]
	Summing up we have that
	\begin{equation}\label{eq:V_tilde_11}
	\int_{\R^N_+}\nabla W\cdot\nabla \phi\dx=0\quad\text{for all }\phi\in C_c^\infty(\overline{\R^N_+}\setminus K).
	\end{equation}
	By weak lower semicontinuity and \eqref{eq:V_tilde_1} we have that
	\begin{equation}\label{eq:V_tilde_10}
	\int_{\R^N_+}|\nabla W|^2\dx<\infty.
	\end{equation}
Now let $R>r(M)$ and $\epsilon<R_0/R$. Thanks to Lemma \ref{lemma:hardy} and estimates
\eqref{eq:5} and \eqref{eq:7}  we have that
	\begin{align*}
\int_{B_R^+}\frac{|\tilde{V}_\epsilon|^2}{\abs{x}^2}\dx&\leq\int_{B_{R_0/\epsilon}^+}\frac{|\tilde{V}_\epsilon|^2}{\abs{x}^2}\dx\\
&\leq 
\frac{2(N+1)}{(N-2)R_0^2}\,\epsilon^2\int_{B_{R_0/\epsilon}^+}\tilde{V}_\epsilon^2 \dx+
\frac{2N}{(N-2)^2}\int_{B_{R_0/\epsilon}^+}
|\nabla\tilde{V}_\epsilon|^2\dx\leq C_4,
\end{align*}
for a certain $C_4>0$ not depending on $\epsilon$ and $R$.
By weak lower semicontinuity we deduce that
$\int_{B_R^+}\frac{|W|^2}{\abs{x}^2}\dx\leq C_4$ for all $R>r(M)$, hence
	\[
	\int_{\R^N_+}\frac{\abs{W}^2}{\abs{x}^2}\dx<\infty.
	\]
	Thanks to this and to \eqref{eq:V_tilde_10}, we have that $W\in\mathcal{D}^{1,2}(\overline{\R^N_+})$; moreover, by density of $C_c^\infty(\overline{\R^N_+}\setminus K)$ in $\mathcal{D}^{1,2}(\overline{\R^N_+}\setminus K)$ we have that \eqref{eq:V_tilde_11} holds for any $\phi\in\mathcal{D}^{1,2}(\overline{\R^N_+}\setminus K)$. 

Let $\eta\in C_c^\infty(\overline{\R^N_+})$ be such that $\eta=1$ in a neighbourhood of $M$: since, for every $R>r(M)$,   $\tilde{V}_\epsilon-\eta\tilde{\phi}_\epsilon\in \HO{B_R^+}{\tilde{K}_\epsilon}$ and since $\HO{B_R^+}{\tilde{K}_\epsilon}$ is converging to $\HO{B_R^+}{K}$ in the sense of Mosco, then, passing to the weak limit, there holds $W-\eta\psi_\gamma\in \HO{B_R^+}{K}$ (see \eqref{eq:phi_0_psi_gamma}). Hence, in view of Lemma \ref{lemma:H_implies_D}, we have that $W-\eta\psi_\gamma\in \mathcal{D}^{1,2}(\overline{\R^N_+}\setminus K)$. Combining this fact with \eqref{eq:V_tilde_11}, by uniqueness (stated in Lemma \ref{lemma:limit_pro}) we have that $W=\tilde{V}$ and, by Urysohn's Subsequence Principle, we conclude that the convergences
\eqref{eq:V_tilde_3},\eqref{eq:V_tilde_4}, and\eqref{eq:V_tilde_5}
 hold as $\epsilon\to 0$ (not only along a sequence $\epsilon_n\to0^+$). 
\end{proof}

Now we are ready for the proof of our second main result.

\begin{proof}[Proof of Theorem \ref{thm:blow_up}]
	Let $R>r(M)$ and let $\eta\in C_c^\infty(\R^N)$ be such that $\eta=1$ in $B_R$. Also let $\tilde{R}>0$ be such that $\supp \eta\sub B_{\tilde{R}}$.
For $\epsilon>0$ small we define
\[
\eta_\epsilon(x)=
\begin{cases}
  \eta\left(\frac1\epsilon\Phi(x)\right),&\text{if }x\in\mathcal U_0,\\
0 ,&\text{if }x\in\R^N\setminus\mathcal U_0,
\end{cases}
\]
and observe that, if $\epsilon$ is sufficiently small, $\eta_\epsilon\in C^1_c(\R^N)$ and $\eta_\epsilon\equiv 1$ in a neighbourhood  of $K_\epsilon$. 
 Testing equation \eqref{eq:weak_V_K_f} with $V_{K_\epsilon,\varphi_0}-\varphi_0\eta_\epsilon$ leads to
	\begin{align*}
\Capa_{\bar{\Omega},c}(K_\epsilon,\varphi_0)&=
			\int_{\Omega}\left(|\nabla V_{K_\epsilon,\varphi_0}|^2+c
|V_{K_\epsilon,\varphi_0}|^2\right)\dx\\
&=
	\int_{\Omega}(\nabla V_{K_\epsilon,\varphi_0} \cdot \nabla (\varphi_0\eta_\epsilon)+c V_{K_\epsilon,\varphi_0} \varphi_0\eta_\epsilon)\dx\\
& =
	\int_{\Omega\cap\mathcal U_0}(\nabla V_{K_\epsilon,\varphi_0} \cdot \nabla (\varphi_0\eta_\epsilon)+c V_{K_\epsilon,\varphi_0} \varphi_0\eta_\epsilon)\dx.
        \end{align*}
Then, by the change of variable $x=\Phi^{-1}(\epsilon y)$, we obtain 
	\begin{equation}\label{eq:8}
\Capa_{\bar{\Omega},c}(K_\epsilon,\varphi_0)=
\epsilon^{2\gamma+N-2}
	\int_{B_{\tilde{R}}^+}\big(A(\epsilon y)\nabla 
\tilde{V}_\epsilon(y)\cdot\nabla (\eta \tilde{\phi}_\epsilon )(y)+\epsilon^2
\hat{c}(\epsilon y)\tilde{V}_\epsilon(y)\eta(y)\tilde{\phi}_\epsilon (y)\big)\dy
\end{equation}
where $\tilde{V}_\epsilon$ and $\tilde{\phi}_\epsilon$ are defined in
\eqref{eq:def_tilde_phi_V}.   
From Proposition \ref{prop:blow_up}
and Proposition \ref{prop:vanish_phi_0}
it follows that 
	\begin{equation}\label{eq:blow_up_1}
\lim_{\epsilon\to0}
\int_{B_{\tilde{R}}^+}\big(A(\epsilon y)\nabla 
\tilde{V}_\epsilon(y)\cdot\nabla (\eta \tilde{\phi}_\epsilon )(y)+\epsilon^2
\hat{c}(\epsilon y)\tilde{V}_\epsilon(y)\eta(y)\tilde{\phi}_\epsilon (y)\big)\dy
=\int_{\R^N_+}\nabla\tilde{V}\cdot\nabla (\eta\psi_\gamma)\dy.
\end{equation}
Finally, testing the equation satisfied by $\tilde{V}$ (see Lemma \ref{lemma:limit_pro}) with $\tilde{V}-\eta\psi_\gamma\in\mathcal{D}^{1,2}(\overline{\R^N_+}\setminus K)$ we have that
	\[
	\mathrm{cap}_{\overline{\R^N_+}}(K,\psi_\gamma)=\int_{\R^N_+}|\nabla \tilde{V}|^2\dx =\int_{\R^N_+}\nabla\tilde{V}\cdot\nabla (\eta\psi_\gamma)\dx.
	\]
	This, combined with \eqref{eq:8} and \eqref{eq:blow_up_1}, concludes the proof.
\end{proof}

\section{Set scaling to an interior point}\label{sec:blow_up_general}

In this last section we consider the case in which the perturbing sets $K_\epsilon$ are concentrating to an interior point in a way that resembles (and comprehends) the scaling of a fixed compact set and we sketch the steps that lead to the proof of Theorem \ref{thm:blow_up_3} (counterpart of Theorem \ref{thm:blow_up}). Always in the case $N\geq 3$, we assume that $0\in\Omega$ and that the family of compact sets $K_\epsilon\sub \Omega$ satisfy 
\eqref{eq:hp_blow_up_1_omega_intr} and \eqref{eq:hp_blow_up_2_omega_intr}.
Heuristically speaking, in the previous section the rescaled domain $\Omega/\epsilon$ was ``approaching'' the half space $\R^N_+$, due to the fact that $0\in\partial\Omega$ and that $\partial\Omega$ was  smooth in a neighbourhood of the origin. In this section, since $0\in\Omega$ the ``limit'' domain of $\Omega/\epsilon$ turns out to be the whole space $\R^N$. For the same reason the role of half balls $B_R^+$ is played, in this section, by balls $B_R$.

Let $\kappa$ and $\zeta_\kappa$ be as in \eqref{eq:psi_hat_gamma}. As in Lemma \ref{lemma:big_O_cap}, by testing $\Capa_{\bar{\Omega},c}(K_\epsilon,\phi_0)$ with $\phi_0$ suitably  cutted off, it is possible to prove that
\begin{equation*}
	\Capa_{\bar{\Omega},c}(K_\epsilon,\phi_0)=O(\epsilon^{N+2\kappa-2}),\quad\text{as }\epsilon\to 0.
\end{equation*}
Also in this framework, a Poincar\'e type inequality holds and the proof follows the same steps as Lemma \ref{lemma:poincare}.
\begin{lemma}[Poincar\'e Inequality]\label{lemma:poincare_1}
	Let $M,K\sub \R^N$ and 
$\{K_\epsilon\}_{\epsilon\in(0,1)}$ satisfy 
\eqref{eq:hp_blow_up_1_omega_intr} and \eqref{eq:hp_blow_up_2_omega_intr} and let 
$\tilde{K}_\epsilon:=K_\epsilon/\epsilon$. Let us assume that $\mathrm{cap}_{\R^N}(K)>0$. For any $R>r(M)$ there exist $\epsilon_0\in(0,1)$ and $C>0$ (both depending on
 $R$ and $K$) such that
	\[
	\int_{B_R}u^2\dx\leq C\int_{B_R}\abs{\nabla u}^2\dx
	\]
	for all $u\in \HO{B_R}{\tilde{K}_\epsilon}$ and for all $\epsilon<\epsilon_0$.
\end{lemma}

Furthermore, the capacity $\capa_{\R^N}(K,\zeta_\kappa)$, whose definition is recalled in Definition \ref{def:newt_capa}, is attained by a potential $\widehat{V}\in\mathcal{D}^{1,2}(\R^N)$, analogously to what is stated in Lemma \ref{lemma:limit_pro}. Also in this context it is possible to prove an Hardy type inequality, which reads as follows.

\begin{lemma}[Hardy-type inequality]\label{lemma:hardy_1}
	We have that
	\[
	\frac{N-2}{2}\int_{B_R}\frac{u^2}{\abs{x}^2}\dx\leq \frac{N+1}{R^2}\int_{B_R}u^2\dx+\frac{N}{N-2}\int_{B_R}\abs{\nabla u}^2\dx
	\]
	for all $u\in H^1(B_R)$ and for all $R>0$.
\end{lemma}

Following the same steps as in the proof of Proposition \ref{prop:blow_up} and Theorem \ref{thm:blow_up} and adapting the ideas and the computations to the current framework, it is possible to prove Theorem \ref{thm:blow_up_3}.

\section*{Acknowledgments}
\noindent 
The authors acknowledge the support of INdAM and  CNRS-PICS project n. PICS08262 entitled ``VALeurs propres d'un op\'erateur Aharonov-Bohm avec p\^oLE variable-VALABLE''.
V. Felli is partially supported by the PRIN 2015
grant ``Variational methods, with applications to problems in
mathematical physics and geometry''.
B. Noris was partially supported by the INdAM-GNAMPA Project 2019 ``Il modello di Born-Infeld per l'elettromagnetismo nonlineare: esistenza, regolarit\`a e molteplicit\`a di soluzioni''.

\end{document}